\documentclass[reqno]{amsart}
\usepackage{mathrsfs}
\usepackage{color}
\usepackage{amsmath}
\usepackage{amsfonts}
\usepackage{amssymb}
\usepackage{graphicx}%
\usepackage{hyperref}

 \newtheorem{Theorem}{Theorem}[section]
 \newtheorem{Corollary}{Corollary}[section]
 \newtheorem{Lemma}{Lemma}[section]
 \newtheorem{Proposition}{Proposition}[section]

 \newtheorem{Remark}{Remark}[section]

 \numberwithin{equation}{section}

\begin{document}

\title[A sharp effectiveness result of SOC]
 {A sharp effectiveness result of
 Demailly's strong openness conjecture}

\author{Qi'an Guan}
\address{Qi'an Guan: School of Mathematical Sciences,
Peking University, Beijing, 100871, China.}
\email{guanqian@math.pku.edu.cn}

\thanks{The author was partially supported by by NSFC-11522101 and NSFC-11431013.}

\subjclass[2010]{32D15, 32E10, 32L10, 32U05, 32W05}

\keywords{strong openness conjecture, multiplier ideal sheaf, sharp effectiveness,
plurisubharmonic function, sublevel set}

\date{\today}

\dedicatory{}

\commby{}

%%% ----------------------------------------------------------------------

\begin{abstract}
In this article,
we establish a sharp effectiveness result of Demailly's strong openness conjecture.
We also establish a sharp effectiveness result related to a conjecture posed by Demailly and Koll\'{a}r.

\end{abstract}

%%% ----------------------------------------------------------------------
\maketitle
%%% ----------------------------------------------------------------------

\section{Introduction}

The multiplier ideal sheaf related to a plurisubharmonic function plays an important role in complex geometry and algebraic geometry,
which was widely discussed
(see e.g. \cite{tian87,Nadel90,siu96,DEL00,D-K01,demailly-note2000,D-P03,Laz04,siu05,siu09,demailly2010}).
We recall the definition as follows.

\emph{Let $\varphi$ be a plurisubharmonic function (see \cite{demailly-book,Si,siu74}) on a complex manifold.
It is known that the multiplier ideal sheaf $\mathcal{I}(\varphi)$ was defined as the sheaf of germs of holomorphic functions $f$ such that
$|f|^{2}e^{-\varphi}$ is locally integrable (see \cite{demailly2010}).}

In \cite{demailly2010} (see also \cite{demailly-note2000}),
Demailly posed the strong openness conjecture for multiplier ideal sheaves (SOC for short),
i.e. $$\mathcal{I}(\varphi)=\mathcal{I}_{+}(\varphi):=\cup_{p>1}\mathcal{I}(p\varphi).$$

The two-dimensional case of SOC was proved by Jonsson-Musta\c{t}\u{a} \cite{JM12}.
An important case of SOC so-called the openness conjecture (OC for short) was proved by Berndtsson \cite{berndtsson13} and the two-dimensional case of OC was proved
by Favre-Jonsson \cite{FM05j,FM05v}.

Recently, SOC was proved in \cite{GZopen-c}
(see also \cite{Lempert14,Hiep14}).
After that,
stimulated by the effectiveness result in Berndtsson's solution of the openness conjecture,
an effectiveness result of SOC was established in \cite{GZopen-effect} as continuous work of the solution of SOC.
Note that the effectiveness result of SOC is not sharp,
then it is natural to ask:

\

\emph{Can one establish a sharp effectiveness result of SOC?}

\

In the following section, we give an affirmative answer to the above question.

One of the innovations in the present article is that,
instead of the single minimal $L^{2}$ integral on the whole domain considered in previous articles (e.g. \cite{GZopen-effect,guan-zhou13p,guan-zhou13ap,guan-zhou12}),
we consider the minimal $L^{2}$ integrals on all sublevel sets $\{\varphi<-t\}$, e.g. the function $G(t)$ (details see Section \ref{sec:minimal}).

\subsection{A sharp effectiveness result of Demailly's strong openness conjecture}
\

Let $F$ be a holomorphic function on pseudoconvex domain $D\subset\mathbb{C}^{n}$ (see \cite{demailly-book}) containing the origin $o\in\mathbb{C}^{n}$,
and let $\varphi$ be a negative plurisubharmonic function on $D$.

Recall that $c_{o}^{F}(\varphi):=\sup\{c\geq0:|F|^{2}e^{-2c\varphi}$ is $L^1$ on a neighborhood of $o\}$  is the jumping number (see \cite{JM13}).
Especially, when $F\equiv 1$, $c_{o}^{F}(\varphi)$ will degenerate to
the complex singularity exponent $c_{o}(\varphi)$ (or log canonical threshold) (see \cite{tian87,Sho92,Ko92,D-K01}, etc.).

If $c_{o}^{F}(\varphi)\neq+\infty$, $C_{F,\mathcal{I}_{+}(2c_{o}^{F}(\varphi)\varphi)_{o}}(D):=\inf\{\int_{D}|\tilde{F}|^{2}|(\tilde{F}-F,o)\in \mathcal{I}_{+}(2c_{o}^{F}(\varphi)\varphi)_{o}$
$\&$ $\tilde{F}\in\mathcal{O}(D)\}.$
If $c_{o}^{F}(\varphi)=+\infty$, $C_{F,\mathcal{I}_{+}(2c_{o}^{F}(\varphi)\varphi)_{o}}(D):=\int_{D} |F|^{2}$.

In this section, we establishing a sharp effectiveness result of SOC.

\begin{Theorem}
\label{thm:effect_optimal}
Assume that $\int_{D} |F|^{2}e^{-\varphi}<+\infty$.
Then for any $p>1$ satisfying
$$\theta(p)>\frac{\int_{D} |F|^{2}e^{-\varphi}}{C_{F,\mathcal{I}_{+}(2c_{o}^{F}(\varphi)\varphi)_{o}}(D)},$$
we have
$$(F,o)\in\mathcal{I}(p\varphi)_{o},$$
i.e. $|F|^{2}e^{-\varphi}$ is locally integrable near $o$,
where $\theta(p)=\frac{p}{p-1}$, which is sharp.
\end{Theorem}

When $D$ is the unit disc $\Delta\subseteq\mathbb{C}$, $F\equiv1$ and $\varphi=\frac{2}{p}\log|z|$,
note that $\int_{D}e^{-\varphi}=\int_{\Delta}\frac{1}{|z|^{2/p}}=\frac{\pi}{1-\frac{1}{p}}$,
and $C_{F,\mathcal{I}_{+}(2c_{o}^{F}(\varphi)\varphi)_{o}}(D)=\pi$,
then it is clear that $\frac{\int_{D}|F|^{2}e^{-\varphi}}{C_{F,\mathcal{I}_{+}(2c_{o}^{F}(\varphi)\varphi)_{o}}(D)}=\frac{p}{p-1}$,
which implies that $\theta(p)=\frac{p}{p-1}$ is sharp.

\begin{Remark}
For the case $\theta(p)=(\frac{1}{(p-1)(2p-1)})^{\frac{1}{p}}$,
the effectiveness result of SOC was established in \cite{GZopen-effect},
which implies (a more precise but non-sharp version of)
Berndtsson's effectiveness result of OC (\cite{berndtsson13}, see also \cite{GZopen-effect}).
\end{Remark}

It follows from inequality \ref{equ:20170915a} that Theorem \ref{thm:effect_optimal} degenerates to the following sharp version of Berndtsson's effectiveness result of OC.

\begin{Corollary}
\label{coro:effect_optimal}
Assume that $\int_{D}e^{-\varphi}<+\infty$,
where $\varphi$ is a negative plurisubharmonic function on pseudoconvex domain $D$.
Then
for any $p>1$ satisfying
\begin{equation}
\frac{p}{p-1}>K_{D}(o)\int_{D}e^{-\varphi},
\end{equation}
we have $e^{-p\varphi}$ is locally integrable near $o$,
where $K_D$ is the Bergman kernel on $D$.
\end{Corollary}

Let $D$ be the unit disc $\Delta\subseteq\mathbb{C}$, and $\varphi=\frac{2}{p}\log|z|$.
Note that $\int_{D}e^{-\varphi}=\int_{\Delta}\frac{1}{|z|^{2/p}}=\frac{\pi}{1-\frac{1}{p}}$,
and $K_{D}(o)=\frac{1}{\pi}$,
then it is clear that $K_{D}(o)\int_{D}e^{-\varphi}=\frac{p}{p-1}$.
Then Corollary \ref{coro:effect_optimal} is sharp.

\subsection{A sharp effectiveness result related to a conjecture posed by Demailly and Koll\'{a}r.}
\

In this section, we present the following sharp effectiveness result related to
a conjecture posed by Demailly and Koll\'{a}r
\begin{Theorem}
\label{thm:lower_optimal}
Let $F$ be a holomorphic function on pseudoconvex domain $D\subset\mathbb{C}^{n}$,
and let $\varphi$ be a negative plurisubharmonic function on $D$.
If $c_{o}^{F}(\varphi)<+\infty$, then
$$\frac{1}{r^{2c_{o}^{F}(\varphi)}}\int_{\{\varphi<\log r\}}|F|^{2}\geq C_{F,\mathcal{I}_{+}(2c_{o}^{F}(\varphi)\varphi)_{o}}(D)>0$$
holds for any $r\in(0,1)$.

Especially, if $C_{F,\mathcal{I}_{+}(2c_{o}^{F}(\varphi)\varphi)_{o}}(D)=+\infty$,
then $\int_{\{\varphi<\log r\}}|F|^{2}=+\infty$ for any $r\in(0,1)$.
\end{Theorem}

Let $D=\Delta\subset\mathbb{C}$ be the unit disc, and let $\varphi=\log|z|$ and $F\equiv 1$.
It is clear that $c_{o}^{F}(\varphi)=1$, $\int_{\{\varphi<\log r\}}|F|^{2}=\pi r^{2}$, $C_{F,\mathcal{I}_{+}(2c_{o}^{F}(\varphi)\varphi)_{o}}(D)=\pi$,
which imply the sharpness of Theorem \ref{thm:lower_optimal}.

When $F\equiv1$, Theorem \ref{thm:lower_optimal} degenerates to

\begin{Corollary}
\label{coro:lower_optimal}
Let $\varphi$ be a negative plurisubharmonic function on pseudoconvex domain $D\subset\mathbb{C}^{n}$.
If $c_{o}^{F}(\varphi)<+\infty$, then
\begin{equation}
\label{equ:effectDK}
\frac{1}{r^{2c_{o}(\varphi)}}\int_{\{\varphi<\log r\}}1\geq K_{D}^{-1}(o)
\end{equation}
holds for any $r\in(0,1)$.
\end{Corollary}

Let $D=\Delta\subset\mathbb{C}$ be the unit disc, and let $\varphi=\log|z|$.
It is clear that $c_{o}(\varphi)=1$, $\int_{\{\varphi<\log r\}}1=\pi r^{2}$, $K_{D}^{-1}(o)=\pi$,
which imply the sharpness of Corollary \ref{coro:lower_optimal}.

In \cite{D-K01} (see also \cite{JM13}), Demailly and Koll\'{a}r conjectured that
$$
\liminf_{r\to0}\frac{1}{r^{2c_{o}(\varphi)}}\int_{\{\varphi<\log r\}}1>0.
$$

Depending on the truth of OC,
the above conjecture was proved in \cite{GZopen-effect} (the two-dimensional case was proved by Favre-Jonsson \cite{FM05j}).
Note that the proof of Theorem \ref{thm:lower_optimal} doesn't depend on the truth of OC,
then we obtain a new approach to the above conjecture
with sharp effectiveness (inequality \ref{equ:effectDK}).

\section{Preparations}\label{sec:preparations}

In this section,
we will do some preparations.

\subsection{Some properties of $C_{f,\varphi}(D)$}\label{sec:minimal}
\

Let $D\subset\mathbb{C}^{n}$ be a pseudoconvex domain.
Let $f$ be a holomorphic function near $o$,
and let $I\subset\mathcal{O}_{o}$ be an ideal.
$C_{f,I}(D)$ denotes $\inf\{\int_{D}|\tilde{f}|^{2}|(\tilde{f}-f,o)\in I{\,}\&{\,}\tilde{f}\in\mathcal{O}(D)\}$
as a generalized version of $C_{F,\mathcal{I}_{+}(2c_{o}^{F}(\varphi)\varphi)_o}(D)$.

If there is no holomorphic function $\tilde{f}$ satisfying both $(\tilde{f}-f,o)\in I$ and $\tilde{f}\in\mathcal{O}(D)$,
then we set $C_{f,I}(D)=-\infty$.
Especially, if $I=\mathcal{I}(\varphi)_{o}$,
then $C_{f,\varphi}(D)$ denotes $C_{f,I}(D)$.

In this section, we will recall and present some properties related to $C_{f,\varphi}(D)$.

\begin{Lemma}
\label{lem:0}
$(f,o)\not\in\mathcal{I}(\varphi)_{o}\Leftrightarrow C_{f,\varphi}(D)\neq0$ (maybe $-\infty$ or $+\infty$).
Especially, if $f\equiv1$ and $\mathcal{I}(\varphi)_{o}\neq\mathcal{O}_{o}$, then $C_{1,\varphi}(D)\geq K_{D}^{-1}(o)$.
\end{Lemma}

Note that $C_{1,\mathcal{I}_{+}(2c_{o}(\varphi)\varphi)_{o}}(D)=C_{1,\mathcal{I}(p_{0}\varphi)_{o}}(D)$
for some $p_{0}>2c_{o}(\varphi)$ (Noetherian of $\mathcal{O}_{o}$),
then it follows from $C_{1,\mathcal{I}(p_{0}\varphi)_{o}}(D)\geq K_{D}(o)$ (Lemma \ref{lem:0}) that
\begin{equation}
\label{equ:20170915a}
C_{1,\mathcal{I}_{+}(2c_{o}(\varphi)\varphi)_{o}}(D)\geq K_{D}(o).
\end{equation}

\begin{proof}(Proof of Lemma \ref{lem:0})
It is clear that $(f,o)\in\mathcal{I}(\varphi)_{o}\Rightarrow C_{f,\varphi}(D)=0$.

Firstly, we prove that $(f,o)\not\in\mathcal{I}(\varphi)_{o}\Rightarrow C_{f,\varphi}(D)\neq 0$ (maybe $-\infty$ or $+\infty$).
We prove it by contradiction:
if not, then there exists holomorphic functions $\{\tilde{f}_{j}\}_{j\in\mathbb{N}^{+}}$ on $D$
such that $\lim_{j\to+\infty}\int_{D}|\tilde{f}_{j}|^{2}=0$ and $(f_{j}-F,o)\in\mathcal{I}(\varphi)_{o}$ for any $j$,
which implies that there exists a subsequence of $\{\tilde{f}_{j}\}_{j\in\mathbb{N}^{+}}$
denoted by $\{\tilde{f}_{j_{k}}\}_{k\in\mathbb{N}^{+}}$ compactly convergent to $0$.
It is clear that $\tilde{f}_{j_{k}}-f$ is compactly convergent to $0-f=f$ near $o$.
It follows from the closedness of the sections of coherent analytic sheaves under the topology of compact convergence (see \cite{G-R})
that $(f,o)\in\mathcal{I}(\varphi)_{o}$,
which contradicts $(f,o)\not\in\mathcal{I}(\varphi)_{o}$. Then we obtain $(f,o)\not\in\mathcal{I}(\varphi)_{o}\Rightarrow C_{f,\varphi}(D)>0$ (maybe $+\infty$).

Secondly, we prove $C_{1,\varphi}(D)\geq K_{D}^{-1}(o)$.
Note that $(\tilde{f}-f,o)\in\mathcal{I}(\varphi)_{o}\neq\mathcal{O}_{o}$ implies $(\tilde{f}-f)(o)=0$ i.e. $\tilde{f}(o)=1$,
then we have $\int_{D}|\tilde{f}|^{2}\geq K_{D}^{-1}(o)$,
which implies $C_{1,\varphi}(D)\geq K_{D}^{-1}(o)$.
Lemma \ref{lem:0} has thus been proved.
\end{proof}

\begin{Lemma}
\label{lem:A}
Let $\varphi$ a negative plurisubharmonic function on $D$,
and let $F$ be a holomorphic function on $\{\varphi<-t\}$.
Assume that $C_{F,\varphi}(\{\varphi<-t\})<+\infty$.
Then there exists a unique holomorphic function $F_{t}$ on
$\{\varphi<-t\}$ satisfying $(F_{t}-F,o)\in\mathcal{I}(\varphi)_{o}$ and $\int_{\{\varphi<-t\}}|F_{t}|^{2}=C_{F,\varphi}(\{\varphi<-t\})$.
Furthermore,
for any holomorphic function $\hat{F}$ on $\{\varphi<-t\}$ satisfying $(\hat{F}-F,o)\in\mathcal{I}(\varphi)_{o}$ and
$\int_{\{\varphi<-t\}}|\hat{F}|^{2}<+\infty$,
we have the following equality
\begin{equation}
\label{equ:20170913e}
\int_{\{\varphi<-t\}}|F_{t}|^{2}+\int_{\{\varphi<-t\}}|\hat{F}-F_{t}|^{2}=\int_{\{\varphi<-t\}}|\hat{F}|^{2}.
\end{equation}
\end{Lemma}

\begin{proof}
Firstly, we prove the existence of $F_{t}$.
As $C_{F,\varphi}(\{\varphi<-t\})<+\infty$
then there exists holomorphic functions $\{f_{j}\}_{j\in\mathbb{N}^{+}}$ on $\{\varphi<-t\}$ such that $\int_{D}|f_{j}|^{2}\to C_{F,\varphi}(\{\varphi<-t\})$,
and $(f_{j}-F,o)\in\mathcal{I}(\varphi)_{o}$.
Then there exists a subsequence of $\{f_{j}\}$ compact convergence to a holomorphic function $f$ on $\{\varphi<-t\}$
satisfying $\int_{K}|f|^{2}\leq C_{F,\varphi}(\{\varphi<-t\})$ for any compact set $K\subset\{\varphi<-t\}$,
which implies $ \int_{D}|f|^{2}\leq C_{F,\varphi}(\{\varphi<-t\})$ by Levi's Theorem.
Note that
the closedness of the sections of coherent analytic sheaves under the topology of compact convergence (see \cite{G-R})
implies that $(f-F,o)\in\mathcal{I}(\varphi)_{o}$,
then we obtain the existence of $F_{t}(=f)$.

Secondly, we prove the uniqueness of $F_{t}$ by contradiction:
if not, there exist two different holomorphic functions $f_{1}$ and $f_{2}$ on on $\{\varphi<-t\}$
satisfying $\int_{\{\varphi<-t\}}|f_{1}|^{2}=\int_{\{\varphi<-t\}}|f_{2}|^{2}=C_{F,\varphi}(\{\varphi<-t\})$,
$(f_{1}-F,o)\in\mathcal{I}(\varphi)_{o}$ and $(f_{2}-F,o)\in\mathcal{I}(\varphi)_{o}$.
Note that
$\int_{\{\varphi<-t\}}|\frac{f_{1}+f_{2}}{2}|^{2}+\int_{\{\varphi<-t\}}|\frac{f_{1}-f_{2}}{2}|^{2}
=\frac{\int_{\{\varphi<-t\}}|f_{1}|^{2}+\int_{\{\varphi<-t\}}|f_{2}|^{2}}{2}=C_{F,\varphi}(\{\varphi<-t\})$,
then we obtain that $\int_{\{\varphi<-t\}}|\frac{f_{1}+f_{2}}{2}|^{2}<C_{F,\varphi}(\{\varphi<-t\})$,
and $(\frac{f_{1}+f_{2}}{2}-F,o)\in\mathcal{I}(\varphi)_{o}$, which contradicts the definition of $C_{F,\varphi}(\{\varphi<-t\})$.

Finally, we prove equality \ref{equ:20170913e}.
For any holomorphic $f$ on $\{\varphi<-t\}$ satisfying $\int_{\{\varphi<-t\}}|f|^{2}<+\infty$ and $(f,o)\in \mathcal{I}(\varphi)_{o}$,
it is clear that
for any complex number $\alpha$,
$F_{t}+\alpha f$ satisfying $((F_{t}+\alpha f)-F,o)\in \mathcal{I}(\varphi)_{o}$,
and $\int_{\{\varphi<-t\}}|F_{t}|^{2}<\int_{\{\varphi<-t\}}|F_{t}+\alpha f|^{2}<+\infty$.
Note that $\int_{\{\varphi<-t\}}|F_{t}+\alpha f|^{2}-\int_{\{\varphi<-t\}}|F_{t}|^{2}>0$ implies $\Re\int_{\{\varphi<-t\}}F_{t}\bar{f}=0$ by considering $\alpha\to0$,
then we obtain $\int_{\{\varphi<-t\}}|F_{t}+f|^{2}=\int_{\{\varphi<-t\}}|F_{t}|^{2}+\int|f|^{2}$.
Choosing $f=\hat{F}-F_{t}$, we obtain equality \ref{equ:20170913e}.

\end{proof}

Let $F$ be a holomorphic function on $D$.
$G(t)$ denotes $C_{F,\varphi}(\{\varphi<-t\})$.
In the following part of this section,
we will consider the properties of $G(t)$.
The following Lemma will be used to prove Proposition \ref{prop:lower_sharp}.

\begin{Lemma}
\label{lem:B}
Assume that $G(0)<+\infty$.
Then $G(t)$ is decreasing with respect to $t\in[0,+\infty)$,
such that
$\lim_{t\to t_{0}+0}G(t)=G(t_{0})$ $(t_{0}\in[0,+\infty))$,
$\lim_{t\to t_{0}-0}G(t)\geq G(t_{0})$ $(t_{0}\in(0,+\infty))$,
and $\lim_{t\to +\infty}G(t)=0$, where $t_{0}\in[0,+\infty)$.
Especially $G(t)$ is lower semi-continuous on $[0,+\infty)$.
\end{Lemma}

\begin{proof}
By the definition of $G(t)$,
it is clear that $G(t)$ is decreasing on $[0,+\infty)$ and $\lim_{t\to t_{0}-0}G(t)\geq G(t_{0})$.
It suffices to prove $\lim_{t\to t_{0}+0}G(t)=G(t_{0}).$
We prove it by contradiction:
if not,
then $\lim_{t\to t_{0}+0}G(t)<G(t_{0})$.

By Lemma \ref{lem:A}, there exists a unique holomorphic function $F_{t}$ on
$\{\varphi<-t\}$ satisfying $(F_{t}-F,o)$ and
$\int_{\{\varphi<-t\}}|F_{t}|^{2}=G(t)$.
Note that $G(t)$ is decreasing implies that
$\int_{\varphi<-t}|F_{t}|^{2}\leq\lim_{t\to t_{0}+0}G(t)$ for any $t<t_{0}$,
then for any compact subset $K$ of $\{\varphi<-t_{0}\}$,
there exists $\{F_{t_{j}}\}$ $(t_{j}\to t_{0}-0,$ as $j\to+\infty)$
uniformly convergent on $K$,
which implies that there exists a subsequence of $\{F_{t_{j}}\}$ (also denoted by $\{F_{t_{j}}\}$) convergent on
any compact subset of $\{\varphi<-t_{0}\}$.

Let $\hat{F}_{t_{0}}:=\lim_{j\to+\infty}F_{t_{j}}$, which is a holomorphic function on $\{\varphi<-t_{0}\}$.
Then it follows from the decreasing property of $G(t)$ that
$\int_{K}|\hat{F}_{t_{0}}|^{2}\leq \lim_{j\to+\infty}\int_{K}|F_{t_{j}}|^{2}\leq
\lim_{j\to+\infty}G(t_{j})\leq\lim_{t\to t_{0}+0}G(t)$ for any
compact set $K\subset \{\varphi<-t_{0}\}$.
It follows from Levi's theorem that
$\int_{D}|\hat{F}_{t_{0}}|^{2}\leq \lim_{t\to t_{0}+0}G(t)$.
Then we obtain that $G_{t_0}\leq\int_{D}|\hat{F}_{t_{0}}|^{2}\leq \lim_{t\to t_{0}+0}G(t)$,
which contradicts $\lim_{t\to t_{0}+0}G(t)<G(t_{0})$.
\end{proof}

We prove Lemma \ref{lem:C} by the following Lemma,
whose various forms already appear in \cite{guan-zhou13p,guan-zhou13ap} etc.:

\begin{Lemma}\label{lem:GZ_sharp}(see  \cite{GZopen-effect}, see also \cite{guan-zhou13p,guan-zhou13ap})
Let $B\in(0,1]$ be arbitrarily given.
Let $D$ be a pseudoconvex domain in
$\mathbb{C}^{n}$ containing $o$.
Let $\varphi$ be a negative plurisubharmonic function
on $D$, such that $\varphi(o)=-\infty$.
Let $F$ be {an $L^{2}$ integrable} holomorphic function on $\{\varphi<-t_{0}\}$.
Then there exists a
holomorphic function $\widetilde{F}$ on $D$, such that,
$$(\widetilde{F}-F,o)\in\mathcal{I}(\varphi)_{o}$$
and
\begin{equation}
\label{equ:GZa}
\begin{split}
&\int_{ D}|\tilde{F}-(1-b_{t_0,B}(\varphi))F|^{2}
\\\leq&(1-e^{-(t_{0}+B)})\int_{D}\frac{1}{B}(\mathbb{I}_{\{-t_{0}-B<t<-t_{0}\}}\circ\varphi)|F|^{2}e^{-\varphi},
\end{split}
\end{equation}
where $\mathbb{I}_{\{-t_{0}-B<t<-t_{0}\}}$ is the character function of set $\{-t_{0}-B<t<-t_{0}\}$,
$b_{t_0,B}(t)=\int_{-\infty}^{t}\frac{1}{B}\mathbb{I}_{\{-t_{0}-B< s<-t_{0}\}}ds$,
and $t_{0}\geq 0$.
\end{Lemma}

Although the $F$ in Lemma 2.1 in \cite{GZopen-effect} is holomorphic on $D$,
in fact the condition that $F$ is an $L^2$ integrable holomorphic on $\{\varphi<-t_{0}\}$ is enough for the proof in \cite{GZopen-effect}
(details see Section \ref{sec:unify}).

Using Lemma \ref{lem:GZ_sharp}, we present the following Lemma,
which will be used to prove Proposition \ref{prop:lower_sharp}.

\begin{Lemma}
\label{lem:C}
Assume that $G(0)<+\infty$.
Then for any $t_{0}\in[0,+\infty)$,
we have
$$G(0)-G(t_{0})\leq (e^{t_{0}}-1)\liminf_{B\to0+0}(-\frac{G(t_{0}+B)-G(t_{0})}{B}).$$
\end{Lemma}

\begin{proof}
By Lemma \ref{lem:A},
there exists a
holomorphic function $F_{t_{0}}$ on $\{\varphi<t_{0}\}$, such that,
$(F_{t_{0}}-F,o)\in\mathcal{I}(\varphi)_{o}$ and
$\int_{\{\varphi<-t_{0}\}}|F_{t_{0}}|^{2}=G(t_{0})$.

It suffices to consider that $\liminf_{B\to0+0}-\frac{G(t_{0}+B)-G(t_{0})}{B}\in(-\infty,0]$
because of the decreasing property of $G(t)$.
Then there exists $B_{j}\to 0+0$ $(j\to+\infty)$ such that
$$\lim_{j\to+\infty}-\frac{G(t_{0}+B_{j})-G(t_{0})}{B_{j}}=\liminf_{B\to0+0}-\frac{G(t_{0}+B)-G(t_{0})}{B}$$
and $\{\frac{G(t_{0}+B_{j})-G(t_{0})}{B_{j}}\}_{j\in\mathbb{N}^{+}}$ is bounded.

By Lemma \ref{lem:GZ_sharp},
it follows that for any $B_{j}$,
there exists
holomorphic function $\tilde{F}_{j}$ on $D$, such that,
$$(\tilde{F}_{j}-F_{t_{0}},o)\in\mathcal{I}(\varphi)_{o}\,\,\,\,(\Rightarrow (\tilde{F}_{j}-F,o)\in\mathcal{I}(\varphi)_{o})$$
and
\begin{equation}
\label{equ:GZc}
\begin{split}
&\int_{ D}|\tilde{F}_{j}-(1-b_{t_{0},B_{j}}(\varphi))F_{t_{0}}|^{2}
\\\leq&(1-e^{-(t_{0}+B_{j})})\int_{D}\frac{1}{B_{j}}(\mathbb{I}_{\{-t_{0}-B_{j}<t<-t_{0}\}}\circ\varphi)|F_{t_{0}}|^{2}e^{-\varphi}
\\\leq&(e^{(t_{0}+B_{j})}-1)\int_{D}\frac{1}{B_{j}}(\mathbb{I}_{\{-t_{0}-B_{j}<t<-t_{0}\}}\circ\varphi)|F_{t_{0}}|^{2},
\end{split}
\end{equation}

Firstly, we will prove that $\int_{ D}|\tilde{F}_{j}|^{2}$ is bounded with respect to $j$.

Note that
\begin{equation}
\label{equ:20170915d}
\begin{split}
\int_{D}\frac{1}{B_{j}}(\mathbb{I}_{\{-t_{0}-B_{j}<t<-t_{0}\}}\circ\varphi)|F_{t_{0}}|^{2}
\leq&\frac{\int_{\{\varphi<-t_{0}\}}|F_{t_{0}}|^{2}-\int_{\{\varphi<-t_{0}-B_{j}\}}|F_{t_{0}}|^{2}}{B_{j}}
\\\leq&\frac{G(t_{0})-G(t_{0}+B_{j})}{B_{j}},
\end{split}
\end{equation}
and
\begin{equation}
\label{equ:GZd}
\begin{split}
&(\int_{ D}|\tilde{F}_{j}-(1-b_{t_{0},B_{j}}(\varphi))F_{t_{0}}|^{2})^{1/2}
\\\geq&(\int_{ D}|\tilde{F}_{j}|^{2})^{1/2}-(\int_{ D}|(1-b_{t_{0},B_{j}}(\varphi))F_{t_{0}}|^{2})^{1/2}
\end{split}
\end{equation}
then it follows from inequality \ref{equ:GZc} that
\begin{equation}
\label{equ:GZe}
\begin{split}
(\int_{ D}|\tilde{F}_{j}|^{2})^{1/2}
\leq&(e^{(t_{0}+B_{j})}-1)(\frac{G(t_{0})-G(t_{0}+B_{j})}{B_{j}})^{1/2}
\\&+(\int_{ D}|(1-b_{t_{0},B_{j}}(\varphi))F_{t_{0}}|^{2})^{1/2}.
\end{split}
\end{equation}
Since $\{\frac{G(t_{0}+B_{j})-G(t_{0})}{B_{j}}\}_{j\in\mathbb{N}^{+}}$ is bounded and $0\leq b_{t_{0},B_{j}}(\varphi)\leq 1$,
then $\int_{ D}|\tilde{F}_{j}|^{2}$ is bounded with respect to $j$.

\

Secondly, we will prove the main result.

Note that $b_{t_0}(\varphi)=1$ on $\{\varphi\geq -t_{0}\}$,
the it follows that
\begin{equation}
\label{equ:20170915b}
\begin{split}
\int_{D}|\tilde{F}_{j}-(1-b_{t_{0},B_{j}}(\varphi))F_{t_{0}}|^{2}
=&\int_{\{\varphi\geq-t_{0}\}}|\tilde{F}_{j}|^{2}+\int_{\{\varphi<-t_{0}\}}|\tilde{F}_{j}-(1-b_{t_{0},B_{j}}(\varphi))F_{t_{0}}|^{2}
\end{split}
\end{equation}

It is clear that
\begin{equation}
\label{equ:20170915c}
\begin{split}
&\int_{\{\varphi <-t_{0}\}}|\tilde{F}_{j}-(1-b_{t_0}(\varphi))F_{t_{0}}|^{2}
\\\geq&((\int_{\{\varphi <-t_{0}\}}|\tilde{F}_{j}-F_{t_{0}}|^{2})^{1/2}-
(\int_{\{\varphi <-t_{0}\}}|b_{t_0,B_{j}}(\varphi)F_{t_{0}}|^{2})^{1/2})^{2}
\\\geq&
\int_{\{\varphi <-t_{0}\}}|\tilde{F}_{j}-F_{t_{0}}|^{2}-
2(\int_{\{\varphi <-t_{0}\}}|\tilde{F}_{j}-F_{t_{0}}|^{2})^{1/2}
(\int_{\{\varphi <-t_{0}\}}|b_{t_0,B_{j}}(\varphi)F_{t_{0}}|^{2})^{1/2}
\\\geq&
\int_{\{\varphi <-t_{0}\}}|\tilde{F}_{j}-F_{t_{0}}|^{2}-2(\int_{\varphi <-t_{0}}|\tilde{F}_{j}-F_{t_{0}}|^{2})^{1/2}
(\int_{\{-t_{0}-B_{j}<\varphi <-t_{0}\}}|F_{t_{0}}|^{2})^{1/2},
\end{split}
\end{equation}
where the last inequality follows from $0\leq b_{t_{0},B_{j}}(\varphi)\leq 1$ and $b_{t_{0},B_{j}}(\varphi)=0$ on $\{\varphi\leq -t_{0}-B_{0}\}$.

Combining equality \ref{equ:20170915b}, inequality \ref{equ:20170915c} and equality \ref{equ:20170913e},
we obtain that
\begin{equation}
\label{equ:20170915e}
\begin{split}
&\int_{D}|\tilde{F}_{j}-(1-b_{t_{0},B_{j}}(\varphi))F_{t_{0}}|^{2}
\\=&\int_{\{\varphi\geq-t_{0}\}}|\tilde{F}_{j}|^{2}+\int_{\{\varphi<-t_{0}\}}|\tilde{F}-(1-b_{t_{0},B_{j}}(\varphi))F_{t_{0}}|^{2}
\\\geq&
\int_{\{\varphi\geq-t_{0}\}}|\tilde{F}_{j}|^{2}+
\int_{\{\varphi <-t_{0}\}}|\tilde{F}_{j}-F_{t_{0}}|^{2}
\\&-2(\int_{\{\varphi <-t_{0}\}}|\tilde{F}_{j}-F_{t_{0}}|^{2})^{1/2}
(\int_{\{-t_{0}-B_{j}<\varphi <-t_{0}\}}|F_{t_{0}}|^{2})^{1/2}.
\\\geq&
\int_{\{\varphi\geq-t_{0}\}}|\tilde{F}_{j}|^{2}+
\int_{\{\varphi <-t_{0}\}}|\tilde{F}_{j}|^{2}-\int_{\{\varphi <-t_{0}\}}|F_{t_{0}}|^{2}
\\&-2(\int_{\{\varphi <-t_{0}\}}|\tilde{F}_{j}-F_{t_{0}}|^{2})^{1/2}
(\int_{\{-t_{0}-B_{j}<\varphi <-t_{0}\}}|F_{t_{0}}|^{2})^{1/2}
\\=&
\int_{D}|\tilde{F}_{j}|^{2}-\int_{\{\varphi <-t_{0}\}}|F_{t_{0}}|^{2}
\\&-2(\int_{\{\varphi <-t_{0}\}}|\tilde{F}_{j}-F_{t_{0}}|^{2})^{1/2}
(\int_{\{-t_{0}-B_{j}<\varphi <-t_{0}\}}|F_{t_{0}}|^{2})^{1/2}.
\end{split}
\end{equation}

It follows from equality \ref{equ:20170913e} that
\begin{equation}
\label{equ:20170917a}
(\int_{\{\varphi <-t_{0}\}}|\tilde{F}_{j}-F_{t_{0}}|^{2})^{1/2}\leq (\int_{\{\varphi <-t_{0}\}}|\tilde{F}_{j}|^{2})^{1/2}
\leq(\int_{ D}|\tilde{F}_{j}|^{2})^{1/2}.
\end{equation}
Since $\int_{ D}|\tilde{F}_{j}|^{2}$ is bounded with respect to $j$,
then it follows from inequality \ref{equ:20170917a} that
$(\int_{\{\varphi <-t_{0}\}}|\tilde{F}_{j}-F_{t_{0}}|^{2})^{1/2}$ is bounded with respect to $j$.
Using the dominated convergence theorem and $\int_{\{\varphi<-t_{0}\}}|F_{t_{0}}|^{2}=G(t_{0})\leq G(0)<+\infty$,
we obtain that
$\lim_{j\to +\infty}\int_{\{-t_{0}-B_{j}<\varphi <-t_{0}\}}|F_{t_{0}}|^{2}d\lambda_{n}=0,$
Then it follows that
$$\lim_{j\to +\infty}(\int_{\{\varphi <-t_{0}\}}|\tilde{F}_{j}-F_{t_{0}}|^{2})^{1/2}
(\int_{\{-t_{0}-B_{j}<\varphi <-t_{0}\}}|F_{t_{0}}|^{2})^{1/2}=0.$$
Combining with inequality \ref{equ:20170915e},
we obtain
\begin{equation}
\label{equ:20170916a}
\begin{split}
\liminf_{j\to+\infty}\int_{D}|\tilde{F}_{j}-(1-b_{t_{0},B_{j}}(\varphi))F_{t_{0}}|^{2}
\geq \int_{D}|\tilde{F}_{j}|^{2}-\int_{\{\varphi <-t_{0}\}}|F_{t_{0}}|^{2}.
\end{split}
\end{equation}

Using inequality \ref{equ:20170915d} (2rd $"\geq"$), inequality \ref{equ:GZc} (3rd $"\geq"$) and inequality \ref{equ:20170916a} (4th $"\geq"$), we obtain
\begin{equation}
\label{equ:20170912b}
\begin{split}
&(e^{t_{0}}-1)\lim_{j\to+\infty}(-\frac{G(t_{0}+B_{j})-G(t_{0})}{B_{j}})
\\=&\lim_{j\to+\infty}(e^{(t_{0}+B_{j})}-1)(-\frac{G(t_{0}+B_{j})-G(t_{0})}{B_{j}})
\\\geq&\liminf_{j\to+\infty}(e^{(t_{0}+B_{j})}-1)\int_{D}\frac{1}{B_{j}}(\mathbb{I}_{\{-t_{0}-B<t<-t_{0}\}}\circ\varphi)|F_{t_{0}}|^{2}
\\\geq&\liminf_{j\to+\infty}\int_{D}|\tilde{F}_{j}-(1-b_{t_{0},B_{j}}(\varphi))F_{t_{0}}|^{2}
\\\geq&\int_{D}|\tilde{F}_{j}|^{2}-\int_{\{\varphi <-t_{0}\}}|F_{t_{0}}|^{2}
\geq G(0)-G(t_{0}).
\end{split}
\end{equation}
Then Lemma \ref{lem:C} has thus been proved.
\end{proof}

The following Lemma will be used to prove Theorem \ref{thm:effect_optimal}.

\begin{Lemma}
\label{lem:D}Let $F$ be a holomorphic function on pseudoconvex domain $D$,
and let $\varphi$ be a negative plurisubharmonic function on $D$.
Assume that $\int_{D}|F|^{2}e^{-\varphi}<+\infty$.
Then
$$\int_{D}|F|^{2}e^{-\varphi}=\int_{-\infty}^{+\infty}(\int_{\{\varphi<-t\}}|F|^{2})e^{t}dt.$$
\end{Lemma}

\begin{proof}
For any $M\in\mathbb{N}^{+}$,
note that $\frac{1}{2^{m}}\sum_{i=1}^{2^{m}M}\mathbb{I}_{\{e^{-\varphi}>\frac{i}{2^m}\}}$ is increasing with respect to
$m$ and convergent to $e^{-\max\{\varphi,-\log{M}\}}\geq0$ $(m\to+\infty)$,
then it follows from Levi's Theorem that
\begin{equation}
\begin{split}
\label{equ:20170914a}
\int_{D}|F|^{2}e^{-\max\{\varphi,-\log M\}}
&=
\lim_{m\to+\infty}\sum_{i=1}^{2^{m}M}\int_{D}|F|^{2}(\frac{1}{2^{m}}\sum_{i=1}^{2^{m}M}\mathbb{I}_{\{e^{-\varphi}>\frac{i}{2^m}\}})
\\&=\lim_{m\to+\infty}\frac{1}{2^{m}}\sum_{i=1}^{2^{m}M}(\int_{D}|F|^{2}\mathbb{I}_{\{e^{-\varphi}>\frac{i}{2^m}\}})
\\&=\lim_{m\to+\infty}\frac{1}{2^{m}}\sum_{i=1}^{2^{m}M}(\int_{\{e^{-\varphi}>\frac{i}{2^m}\}}|F|^{2}),
\end{split}
\end{equation}
where $\mathbb{I}_{A}$ is the character function of set $A$.

As $\int_{D}|F|^{2}\leq\int_{D}|F|^{2}e^{-\varphi}<+\infty$,
then $\int_{\{e^{-\varphi}>s\}}|F|^{2}$ is finite and non-negative for any $s$ and decreasing with respect to $s$,
which implies $\int_{\{e^{-\varphi}>s\}}|F|^{2}$ is Riemann integrable
and
\begin{equation}
\begin{split}
\label{equ:20170914b}
\int_{0}^{M}(\int_{\{e^{-\varphi}>s\}}|F|^{2})ds
=\lim_{m\to+\infty}\frac{1}{2^{m}}\sum_{i=1}^{2^{m}M}(\int_{\{e^{-\varphi}>\frac{i}{2^m}\}}|F|^{2}),
\end{split}
\end{equation}
holds for any $M\in\mathbb{N}^{+}$.
Note that $\int_{D}|F|^{2}e^{-\varphi}<+\infty$,
then it follows from equality \ref{equ:20170914a} and \ref{equ:20170914b} that
\begin{equation}
\label{equ:20170915g}
\begin{split}
\int_{D}|F|^{2}e^{-\max\{\varphi,-\log M\}}
=\int_{0}^{M}(\int_{\{e^{-\varphi}>s\}}|F|^{2})ds.
\end{split}
\end{equation}
As $\int_{D}|F|^{2}e^{-\varphi}<+\infty$,
then it follows from the Levi's theorem that
\begin{equation}
\label{equ:20170915h}
\begin{split}
\lim_{M\to+\infty}\int_{D}|F|^{2}e^{-\max\{\varphi,-\log M\}}
=\int_{D}|F|^{2}e^{-\varphi}<+\infty.
\end{split}
\end{equation}
Combining equality \ref{equ:20170915g} and equality \ref{equ:20170915h},
we obtain that
\begin{equation}
\label{equ:20170915i}
\begin{split}
\lim_{M\to+\infty}\int_{0}^{M}(\int_{\{e^{-\varphi}>s\}}|F|^{2})ds
&=\lim_{M\to+\infty}\int_{D}|F|^{2}e^{-\max\{\varphi,-\log M\}}
\\&=\int_{D}|F|^{2}e^{-\varphi}<+\infty.
\end{split}
\end{equation}
Note that $\int_{0}^{+\infty}(\int_{\{e^{-\varphi}>s\}}|F|^{2})ds=\lim_{M\to+\infty}\int_{0}^{M}(\int_{\{e^{-\varphi}>s\}}|F|^{2})ds<+\infty$,
then it follows from equality \ref{equ:20170915i} that
\begin{equation}
\label{equ:20170917b}
\int_{0}^{+\infty}(\int_{\{e^{-\varphi}>s\}}|F|^{2})ds=\int_{D}|F|^{2}e^{-\varphi}.
\end{equation}

Let $s=e^{t}$, then
\begin{equation}
\label{equ:20170914c}
\begin{split}
\int_{0}^{+\infty}(\int_{\{e^{-\varphi}>s\}}|F|^{2})ds
&=\int_{-\infty}^{+\infty}(\int_{\{e^{-\varphi}>e^{t}\}}|F|^{2})de^{t}
\\&=\int_{-\infty}^{+\infty}(\int_{\{\varphi<-t\}}|F|^{2})e^{t}dt.
\end{split}
\end{equation}
Combining equality \ref{equ:20170917b} and equality \ref{equ:20170914c},
we obtain Lemma \ref{lem:D}.
\end{proof}

\subsection{A sharp lower bound of the volume of the sublevel sets of plurisubharmonic functions
with a multiplier}
\

In order to prove Theorem \ref{thm:effect_optimal} and Theorem \ref{thm:lower_optimal},
we present the following sharp lower bound of the volume of the sublevel sets of plurisubharmonic functions
with a multiplier.

\begin{Proposition}
\label{prop:lower_sharp}
Let $F$ be a holomorphic function on $D$,
and let $\varphi$ be a negative plurisubharmonic function $\varphi$ on $D$.
Then the inequality
\begin{equation}
\label{equ:lower_sharp}
\int_{\{\varphi<-t\}}|F|^{2}\geq e^{-t}C_{F,\varphi}(D)
\end{equation}
holds for any $t\geq 0$,
which is sharp.

Especially, if $C_{F,\varphi}(D)=+\infty$,
then $\int_{\{\varphi<-t\}}|F|^{2}=+\infty$ for any $t\geq 0$.
\end{Proposition}

Let $D=\Delta\subset\mathbb{C}$ be the unit disc, and let $\varphi=2\log|z|$ and $F\equiv 1$.
It is clear that $C_{F,\varphi}(D)=\pi$, and $\int_{\{\varphi<-t\}}|F|^{2}=e^{-t}\pi$,
which gives the sharpness of Proposition \ref{prop:lower_sharp}.

\begin{proof}
We prove Proposition \ref{prop:lower_sharp} in two steps,
i.e. the case $C_{F,\varphi}(D)<+\infty$ and the case $C_{F,\varphi}(D)=+\infty$.

Step 1. We prove the case $C_{F,\varphi}(D)<+\infty$
As $\int_{\{\varphi<-t\}}|F|^{2}\geq G(t)$ for any $t\in[0,+\infty)$,
then it suffices to prove that
$G(t)\geq e^{-t}G(0)$ for any $t\in[0,+\infty)$.

Let $H(t):=G(t)-e^{-t}G(0).$
We prove $H(t)\geq 0$ by contradiction:
if not, then there exists $t$ such that $H(t)<0$.

Note that $G(t)\in[0,G(0)]$ is bounded on $[0,\infty)$,
then $H(t)$ is also bounded on $[0,\infty)$,
which implies that $\inf_{[0,+\infty)}H(t)$ is finite.

By Lemma \ref{lem:B}, it is clear that
$\lim_{t\to 0+0}H(t)=H(0)=0$
and $\lim_{t\to +\infty}H(t)=\lim_{t\to +\infty}G(t)-\lim_{t\to +\infty}e^{-t}G(0)=0-0=0$.
Then it follows from $\inf_{[0,+\infty)}H(t)<0$ that there exists a closed interval $[a,b]\subset\subset(0,\infty)$
such that $\inf_{[a,b]}H(t)=\inf_{[0,+\infty)}H(t)$.
Since $G(t)$ is lower semi-continuous (Lemma \ref{lem:B}) and $e^{-t}G(0)$ is continuous,
then it follows that $H(t)$ is lower semi-continuous,
which implies that there exists $t_{0}\in[a,b]$ such that $H(t_{0})=\inf_{[0,+\infty)}H(t)<0$.

In the following part of Step 1,
we will consider the negativeness of $(e^{t_{0}}-1)\liminf_{t\to t_{0}+0}(-\frac{H(t)-H(t_{0})}{t-t_{0}})+H(t_{0})$ and
get a contradiction.

As $H(t_{0})=\inf_{[0,+\infty)}H(t)$,
then it follows that $\liminf_{t\to t_{0}+0}(-\frac{H(t)-H(t_{0})}{t-t_{0}})=-\limsup_{t\to t_{0}+0}\frac{H(t)-H(t_{0})}{t-t_{0}}\leq 0$.
Combining with $H(t_{0})<0$,
then we obtain that
\begin{equation}
\label{equ:20170916b}
(e^{t_{0}}-1)\liminf_{t\to t_{0}+0}(-\frac{H(t)-H(t_{0})}{t-t_{0}})+H(t_{0})<0.
\end{equation}

Note that
\begin{equation}
\label{equ:20170912f}
\begin{split}
&(e^{t_{0}}-1)\liminf_{t\to t_{0}+0}(-\frac{H(t)-H(t_{0})}{t-t_{0}})+H(t_{0})
\\=&
(e^{t_{0}}-1)\liminf_{t\to t_{0}+0}(-\frac{(G(t)-e^{-t}G(0))-(G(t_{0})-e^{-t_{0}}G(0))}{t-t_{0}})
\\&+(G(t_{0})-e^{-t_{0}}G(0))
\\=&(e^{t_{0}}-1)\liminf_{t\to t_{0}+0}(-\frac{G(t)-G(t_{0})}{t-t_{0}})
\\&+(e^{t_{0}}-1)\lim_{t\to t_{0}+0}(\frac{e^{-t}G(0)-e^{-t_{0}}G(0)}{t-t_{0}})+(G(t_{0})-e^{-t_{0}}G(0))
\\=&(e^{t_{0}}-1)\liminf_{t\to t_{0}+0}(-\frac{G(t)-G(t_{0})}{t-t_{0}})
\\&-(e^{t_{0}}-1)e^{-t_{0}}G(0)+(G(t_{0})-e^{-t_{0}}G(0))
\\=&(e^{t_{0}}-1)\liminf_{t\to t_{0}+0}(-\frac{G(t)-G(t_{0})}{t-t_{0}})+G(t_{0})-G(0),
\end{split}
\end{equation}
then it follows from Lemma \ref{lem:C}
that
\begin{equation}
\label{equ:20170912g}
\begin{split}
(e^{t_{0}}-1)\liminf_{t\to t_{0}+0}(-\frac{H(t)-H(t_{0})}{t-t_{0}})+H(t_{0})\geq0,
\end{split}
\end{equation}
which contradicts inequality \ref{equ:20170916b}.
The Case $C_{F,\varphi}(D)<+\infty$ has thus been proved.

\

Step 2.  We prove the case $C_{F,\varphi}(D)=+\infty$ by contradiction:
if not, then integral $\int_{\{\varphi<-t_{0}\}}|F|^{2}$ is finite for some $t_{0}\geq 0$.
It follows from Lemma \ref{lem:GZ_sharp} that for $B=1$, there exists holomorphic function $\tilde{F}$ on $D$
satisfying
$$(\widetilde{F}-F,o)\in\mathcal{I}(\varphi)_{o}$$
and
\begin{equation}
\label{equ:20170913f}
\begin{split}
&\int_{ D}|\tilde{F}-(1-b_{t_0,B}(\varphi))F|^{2}
\\\leq&(e^{(t_{0}+B)}-1)\int_{D}\frac{1}{B}(\mathbb{I}_{\{-t_{0}-B<t<-t_{0}\}}\circ\varphi)|F|^{2}.
\end{split}
\end{equation}
Note that
\begin{equation}
\label{equ:20170913g}
\begin{split}
&\int_{ D}|\tilde{F}-(1-b_{t_{0},B}(\varphi))F|^{2}
\\\geq&((\int_{ D}|\tilde{F}|^{2})^{1/2}-(\int_{ D}|(1-b_{t_{0},B}(\varphi))F|^{2})^{1/2})^{2},
\end{split}
\end{equation}
then it follows from inequality \ref{equ:20170912f} that
\begin{equation}
\label{equ:20170913h}
\begin{split}
&(e^{(t_{0}+B)}-1)\int_{D}\frac{1}{B}(\mathbb{I}_{\{-t_{0}-B<t<-t_{0}\}}\circ\varphi)|F|^{2}
\\\geq&((\int_{ D}|\tilde{F}|^{2})^{1/2}-(\int_{ D}|(1-b_{t_{0},B}(\varphi))F|^{2})^{1/2})^{2},
\end{split}
\end{equation}
which implies
\begin{equation}
\label{equ:20170913i}
\begin{split}
&((e^{(t_{0}+B)}-1)\int_{D}\frac{1}{B}(\mathbb{I}_{\{-t_{0}-B<t<-t_{0}\}}\circ\varphi)|F|^{2})^{1/2}
\\&+(\int_{ D}|(1-b_{t_{0},B}(\varphi))F|^{2})^{1/2}\geq((\int_{ D}|\tilde{F}|^{2})^{1/2}.
\end{split}
\end{equation}
Since $b_{t_{0},B}(\varphi)=1$ on $\{\varphi\geq t_{0}\}$, $0\leq b_{t_{0},B}(\varphi)\leq 1$
and $\int_{\{\varphi<-t_{0}\}}|F|^{2}<+\infty$,
then we obtain that $\int_{D}|(1-b_{t_{0},B}(\varphi))F|^{2}=\int_{\{\varphi<-t_{0}\}}|(1-b_{t_{0},B}(\varphi))F|^{2}\leq\int_{\{\varphi<-t_{0}\}}|F|^{2}<+\infty$ and
$\int_{D}\frac{1}{B}(\mathbb{I}_{\{-t_{0}-B<t<-t_{0}\}}\circ\varphi)|F|^{2}<+\infty$.
It is clear that the LHS of inequality \ref{equ:20170913i} is finite, which implies that the RHS of inequality \ref{equ:20170913i} is also finite.
Then we obtain that $C_{F,\varphi}(D)\leq\int_{ D}|\tilde{F}|^{2}<+\infty$,
which contradicts $C_{F,\varphi}(D)=+\infty$.
The case $C_{F,\varphi}(D)=+\infty$ has thus been proved.
\end{proof}

\section{Proofs of  Theorem \ref{thm:effect_optimal} and Theorem \ref{thm:lower_optimal}}

In this section, we prove Theorem \ref{thm:effect_optimal} and Theorem \ref{thm:lower_optimal}.

\subsection{Proof of Theorem \ref{thm:effect_optimal}}
\

It suffices to consider the case $c_{o}^{F}(\varphi)\neq+\infty$.

Firstly, we prove the following proposition.

\begin{Proposition}
\label{p:effect_optimal}
Let $F$ be a holomorphic function on pseudoconvex domain $D$,
and let $\psi$ be a negative plurisubharmonic function on $D$.
Assume that $C_{F,\psi}(D)\in(0,+\infty]$.
Then for any $p>1$, we have
\begin{equation}
\label{equ:20170914d}
\int_{D}|F|^{2}e^{-\frac{\psi}{p}}\geq \frac{p}{p-1}C_{F,\psi}(D).
\end{equation}
\end{Proposition}

\begin{proof}
If $C_{F,\psi}(D)=+\infty$,
then it is clear that
$\int_{D}|F|^{2}e^{-\frac{\psi}{p}}\geq C_{F,\psi}(D)=+\infty$.

It suffices to consider the situation that $C_{F,\psi}(D)\in(0,+\infty)$ and $\int_{D}|F|^{2}e^{-\frac{\psi}{p}}<+\infty$ both hold.
Then Lemma \ref{lem:D} $(\varphi\sim\frac{\psi}{p})$ implies that
\begin{equation}
\label{equ:20170914e}
\int_{D}|F|^{2}e^{-\frac{\psi}{p}}=\int_{-\infty}^{+\infty}(\int_{\{\frac{\psi}{p}<-t\}}|F|^{2})e^{t}dt.
\end{equation}

It follows from $C_{F,\psi}(D)\in(0,+\infty)$ and Proposition \ref{prop:lower_sharp} that for $t\geq0$,
$$\int_{\{\frac{\psi}{p}<-t\}}|F|^{2}=\int_{\{\psi<-pt\}}|F|^{2}\geq e^{-pt}C_{F,\psi}(D).$$
Then we obtain
\begin{equation}
\label{equ:20170914f}
\int_{0}^{+\infty}(\int_{\{\frac{\psi}{p}<-t\}}|F|^{2})e^{t}dt\geq
\int_{0}^{+\infty}e^{-pt}e^{t}C_{F,\psi}(D)=\frac{1}{p-1}C_{F,\psi}(D).
\end{equation}

As $\int_{\{\frac{\psi}{p}<-t\}}|F|^{2}\geq C_{F,\psi}(D)$ holds for any $t<0$,
then it is clear that
\begin{equation}
\label{equ:20170914g}
\int_{-\infty}^{0}(\int_{\{\frac{\psi}{p}<-t\}}|F|^{2})e^{t}dt\geq C_{F,\psi}(D)\int_{-\infty}^{0}e^{t}dt=C_{F,\psi}(D)
\end{equation}

Combining equality \ref{equ:20170914e}, inequality \ref{equ:20170914f} and inequality \ref{equ:20170914g},
we obtain Proposition \ref{p:effect_optimal}.
\end{proof}

In the following part, we prove Theorem \ref{thm:effect_optimal} by using Proposition \ref{p:effect_optimal}.

Let $\psi=p\varphi$.
If $p>2c_{o}^{F}(\varphi)$,
then it follows from $C_{F,p\varphi}(D)\geq C_{F,\mathcal{I}_{+}(2c_{o}^{F}(\varphi)\varphi)_{o}}(D)>0$ (maybe $+\infty$) and Proposition \ref{p:effect_optimal} that
\begin{equation}
\label{equ:20170914h}
\int_{D}|F|^{2}e^{-\varphi}\geq \frac{p}{p-1}C_{F,\mathcal{I}_{+}(2c_{o}^{F}(\varphi)\varphi)_{o}}(D).
\end{equation}
Taking $p\to2c_{o}^{F}(\varphi)+0$, it is clear that equality \ref{equ:20170914h} also holds for $p\geq 2c_{o}^{F}(\varphi)$.
Then we obtain that if
$$\int_{D}|F|^{2}e^{-\varphi}<\frac{p}{p-1}C_{F,\mathcal{I}_{+}(2c_{o}^{F}(\varphi)\varphi)_{o}}(D),$$
then $p<2c_{o}^{F}(\varphi)$, i.e. $(F,o)\in\mathcal{I}(p\varphi)_{o}$.
Theorem \ref{thm:effect_optimal} has thus been proved.

\subsection{Proof of Theorem \ref{thm:lower_optimal}}
\

We prove Theorem \ref{thm:lower_optimal} for the cases $C_{F,\mathcal{I}_{+}(2c_{o}^{F}(\varphi)\varphi)_{o}}(D)<+\infty$
and $=+\infty$ respectively.

Step 1.  We prove the case $C_{F,\mathcal{I}_{+}(2c_{o}^{F}(\varphi)\varphi)_{o}}(D)<+\infty$.

Proposition \ref{prop:lower_sharp} shows that
\begin{equation}
\label{equ:20170913a}
\int_{\{p\varphi<-t\}}|F|^{2}\geq e^{-t}C_{F,p\varphi}(D)
\end{equation}
holds for any $t\geq 0$ and $p>2c_{o}^{F}(\varphi)$.

By the Noetherian property of $\mathcal{O}_{o}$,
it follows that there exists $p_{0}>2c_{o}^{F}(\varphi)$,
such that $\mathcal{I}_{+}(2c_{o}^{F}(\varphi)\varphi)_{o}=\mathcal{I}(p_{0}\varphi)_{o}$,
which implies
\begin{equation}
\label{equ:20170913b}
\lim_{p\to2c_{o}^{F}(\varphi)+0}C_{F,p\varphi}(D)=C_{F,p_{0}\varphi}(D)=C_{F,\mathcal{I}_{+}(2c_{o}^{F}(\varphi)\varphi)_{o}}(D).
\end{equation}
It follows from Lemma \ref{lem:0} that $C_{F,p_{0}\varphi}(D)>0$ $(\text{maybe} +\infty)$.
Combining with equality \ref{equ:20170913b}, we obtain
\begin{equation}
\label{equ:20170916c}
\lim_{p\to2c_{o}^{F}(\varphi)+0}C_{F,p\varphi}(D)=C_{F,\mathcal{I}_{+}(2c_{o}^{F}(\varphi)\varphi)_{o}}(D)>0\,\,(\text{maybe} +\infty)
\end{equation}

It follows from Levi's theorem
that
\begin{equation}
\label{equ:20170913c}
\lim_{p\to2c_{o}^{F}(\varphi)+0}\int_{\{p\varphi<-t\}}|F|^{2}=\int_{\{2c_{o}^{F}(\varphi)\varphi{\leq}-t\}}|F|^{2}
\end{equation}
holds for any $t\geq0$.

Combining equality \ref{equ:20170913c}, equality \ref{equ:20170913a}, and equality \ref{equ:20170916c},
we obtain
\begin{equation}
\label{equ:20170913d}
\begin{split}
\int_{\{2c_{o}^{F}(\varphi)\varphi{\leq}-t\}}|F|^{2}
&=\lim_{p\to2c_{o}^{F}(\varphi)+0}\int_{\{p\varphi<-t\}}|F|^{2}
\\&=\lim_{p\to2c_{o}^{F}(\varphi)+0}e^{-t}C_{F,p\varphi}(D)
\\&\geq e^{-t}C_{F,\mathcal{I}_{+}(2c_{o}^{F}(\varphi)\varphi)_{o}}(D)>0, (\text{maybe} +\infty).
\end{split}
\end{equation}
{Note that $\{2c_{o}^{F}(\varphi)\varphi<-t\}=\cup_{t'>t}\{2c_{o}^{F}(\varphi)\varphi\leq-t'\}$,
then it follows from inequality \ref{equ:20170913d} that
$\int_{\{2c_{o}^{F}(\varphi)\varphi<-t\}}|F|^{2}\geq\sup_{\{t'>t\}}e^{-t'}C_{F,\mathcal{I}_{+}(2c_{o}^{F}(\varphi)\varphi)_{o}}(D)>0, (\text{maybe} +\infty)$.}
Let $r=e^{-\frac{t}{c_{o}^{F}(\varphi)}}$,
then the case
$C_{F,\mathcal{I}_{+}(2c_{o}^{F}(\varphi)\varphi)_{o}}(D)<+\infty$ has thus been proved.

\

Step 2. We prove the case $C_{F,\mathcal{I}_{+}(2c_{o}^{F}(\varphi)\varphi)_{o}}(D)=+\infty$.
By the Noetherian property of $\mathcal{O}_{o}$,
it follows that there exists $p_{0}>2c_{o}^{F}(\varphi)$,
such that $\mathcal{I}_{+}(2c_{o}^{F}(\varphi)\varphi)_{o}=\mathcal{I}(p_{0}\varphi)_{o}$.
Then it is clear that $C_{F,p_{0}\varphi}(D)=C_{F,\mathcal{I}_{+}(2c_{o}^{F}(\varphi)\varphi)_{o}}(D)=+\infty$.
It follows from Proposition \ref{prop:lower_sharp} that for any $t\in[0,+\infty)$,
$\int_{\{p_{0}\varphi<-t\}}|F|^{2}=+\infty$,
which implies that $\int_{\{\varphi<-t\}}|F|^{2}=+\infty$ for any $t\in[0,+\infty)$.
Then the case $C_{F,\mathcal{I}_{+}(2c_{o}^{F}(\varphi)\varphi)_{o}}(D)=+\infty$ has thus been proved.

\section{Appendix: concavity of minimal $L^{2}$ integrals on sublevel sets of plurisubharmonic functions}

In section \ref{sec:preparations},
we consider the minimal $L^{2}$ integrals $G(t)$ on the sublevel sets of the weights related to multiplier ideals, and obtain that $G(t)\geq e^{-t}G(0)$.

In the present section, we consider a generalized version of $G(t)$ (details see \ref{equ:20171012}), and reveal
the concavity of $G(-\log r)$,
which was contained in section \ref{sec:preparations}.

\begin{Proposition}
\label{prop:logconcave}If $G(0)<+\infty$,
then
$G(-\log r)$ is concave with respect to $r\in(0,1]$.
\end{Proposition}

Especially, the above result is a generalization of $G(t)\geq e^{-t}G(0)$.

Choosing $\psi$ as the polar function $\psi+\log|z_{n}|^{2}$ in \cite{ohsawa3} (see also \cite{guan-zhou12}),
$G(t)$ is the minimal $L^{2}$ extension with negligible weight on $\{\psi+\log|z_{n}|^{2}<-t\}$ in \cite{ohsawa3} (see also \cite{guan-zhou12}).

\subsection{Some results contained in section \ref{sec:preparations}}

In the following part,
we recall and reveal some results contained section \ref{sec:preparations}.

Let $D\subset\mathbb{C}^{n}$ be a pseudoconvex domain containing $o\in\mathbb{C}^{n}$,
and let $\varphi$ be a locally upper bounded Lebesgue measurable function on $D$.
Let $f$ be a holomorphic function near $o$,
and let $I\subset\mathcal{O}_{o}$ be an ideal.
$C_{f,I}(D,\varphi)$ denotes $\inf\{\int_{D}|\tilde{f}|^{2}e^{-\varphi}|(\tilde{f}-f,o)\in I{\,}\&{\,}\tilde{f}\in\mathcal{O}(D)\}$.

If there is no holomorphic function $\tilde{f}$ satisfying both $(\tilde{f}-f,o)\in I$ and $\tilde{f}\in\mathcal{O}(D)$,
then we set $C_{f,I}(D,\varphi)=-\infty$.
Especially, if $I=\mathcal{I}(\psi)_{o}$,
then $C_{f,\psi}(D,\varphi)$ denotes $C_{f,I}(D,\varphi)$.

The proof of Lemma \ref{lem:0} contains

\begin{Lemma}
\label{lem:0a}
$(f,o)\not\in\mathcal{I}(\psi)_{o}\Leftrightarrow C_{f,\psi}(D,\varphi)\neq0$ (maybe $-\infty$ or $+\infty$).
\end{Lemma}

The proof of Lemma \ref{lem:A} contains

\begin{Lemma}
\label{lem:Aa}
Let $\varphi+\psi$ and $\psi<0$ be plurisubharmonic functions on $D$,
and let $F$ be a holomorphic function on $\{\psi<-t\}$.
Assume that $C_{F,\varphi+\psi}(\{\psi<-t\},\varphi)<+\infty$.
Then there exists a unique holomorphic function $F_{t}$ on
$\{\psi<-t\}$ satisfying $(F_{t}-F,o)\in\mathcal{I}(\varphi+\psi)_{o}$ and $\int_{\{\psi<-t\}}|F_{t}|^{2}e^{-\varphi}=C_{F,\varphi+\psi}(\{\psi<-t\},\varphi)$.
Furthermore,
for any holomorphic function $\hat{F}$ on $\{\psi<-t\}$ satisfying $(\hat{F}-F,o)\in\mathcal{I}(\varphi+\psi)_{o}$ and
$\int_{\{\psi<-t\}}|\hat{F}|^{2}e^{-\varphi}<+\infty$,
we have the following equality
\begin{equation}
\int_{\{\psi<-t\}}|F_{t}|^{2}e^{-\varphi}+\int_{\{\psi<-t\}}|\hat{F}-F_{t}|^{2}e^{-\varphi}=\int_{\{\psi<-t\}}|\hat{F}|^{2}e^{-\varphi}.
\end{equation}
\end{Lemma}

Let $F$ be a holomorphic function on $D$,
and let $\varphi+\psi$ and $\psi<0$ be plurisubharmonic functions on $D$.
Denote that
\begin{equation}
\label{equ:20171012}
G(t):=C_{F,\varphi+\psi}(\{\psi<-t\},\varphi).
\end{equation}
The proof of Lemma \ref{lem:B} contains

\begin{Lemma}
\label{lem:Ba}
Assume that $G(0)<+\infty$.
Then $G(t)$ is decreasing with respect to $t\in[0,+\infty)$,
such that
$\lim_{t\to t_{0}+0}G(t)=G(t_{0})$ $(t_{0}\in[0,+\infty))$,
$\lim_{t\to t_{0}-0}G(t)\geq G(t_{0})$ $(t_{0}\in(0,+\infty))$,
and $\lim_{t\to +\infty}G(t)=0$, where $t_{0}\in[0,+\infty)$.
Especially $G(t)$ is lower semi-continuous on $[0,+\infty)$.
\end{Lemma}

{The proof of Lemma 2.1 in \cite{GZopen-effect},
implies the following Lemma (details see section \ref{sec:unify}), whose various forms already appear in \cite{guan-zhou13p,guan-zhou13ap}.}

\begin{Lemma} \label{lem:GZ_sharpA}(see \cite{guan-zhou13p,guan-zhou13ap})
Let $B\in(0,1]$ be arbitrarily given.
Let $D$ be a pseudoconvex domain in
$\mathbb{C}^{n}$ containing $o$.
Let $\psi$ be a negative plurisubharmonic function
on $D$, such that $\psi(o)=-\infty$.
Let $\varphi$ be a locally upper bounded function
on $D$, such that $\varphi+\psi$ is plurisubharmonic on $D$.
Let $F$ be a holomorphic function on $\{\psi<-t_{0}\}$ {such that
$\int_{\{\psi<-t_{0}\}}|F|^{2}e^{-\varphi}<+\infty$.}
Then there exists a
holomorphic function $\widetilde{F}$ on $D$, such that,
$$(\widetilde{F}-F,o)\in\mathcal{I}(\varphi+\psi)_{o}$$
and
\begin{equation}
\begin{split}
&\int_{ D}|\tilde{F}-(1-b_{t_0,B}(\psi))F|^{2}e^{-\varphi}
\\\leq&(1-e^{-(t_{0}+B)})\int_{D}\frac{1}{B}(\mathbb{I}_{\{-t_{0}-B<t<-t_{0}\}}\circ\psi)|F|^{2}e^{-\varphi-\psi},
\end{split}
\end{equation}
where $\mathbb{I}_{\{-t_{0}-B<t<-t_{0}\}}$ is the character function of set $\{-t_{0}-B<t<-t_{0}\}$,
$b_{t_0,B}(t)=\int_{-\infty}^{t}\frac{1}{B}\mathbb{I}_{\{-t_{0}-B< s<-t_{0}\}}ds$,
and $t_{0}\geq 0$.
\end{Lemma}

It follows from Lemma \ref{lem:GZ_sharpA} that the proof of Lemma {\ref{lem:C}} contains

\begin{Lemma}
\label{lem:Ca}
Assume that $G(0)<+\infty$.
Then for any $t_{0}\in[0,+\infty)$,
we have
$$G(0)-G(t_{0})\leq (e^{t_{0}}-1)\liminf_{B\to0+0}(-\frac{G(t_{0}+B)-G(t_{0})}{B}).$$
\end{Lemma}

By replacing $D$ with {the component of} $\{\psi<-t_{1}\}$ {containing $o$} and replacing $\psi$ with $\psi+t_{1}$,
it follows from Lemma \ref{lem:GZ_sharpA} that the proof of Lemma {\ref{lem:C}} contains

\begin{Lemma}
\label{lem:Da}
Assume that $G(0)<+\infty$.
Then for any $t_{0},t_{1}\in[0,+\infty)$,
we have
$$G(t_{1})-G(t_{1}+t_{0})\leq (e^{t_{0}}-1)\liminf_{B\to0+0}(-\frac{G(t_{0}+t_{1}+B)-G(t_{0}+t_{1})}{B}),$$
i.e.
\begin{equation}
\label{equ:20171012b}
\frac{G(t_{1})-G(t_{1}+t_{0})}{e^{-t_{1}}-e^{-(t_{0}+t_{1})}}\leq \liminf_{B\to0+0}
\frac{G(t_{0}+t_{1}+B)-G(t_{0}+t_{1})}{e^{-(t_{0}+t_{1}+B)}-e^{-(t_{0}+t_{1})}}
\end{equation}
\end{Lemma}

As $G(-\log r)$ is lower semicontinuous (Lemma \ref{lem:Ba}),
then it follows from the following well-known property of concave functions (Lemma \ref{lem:Ea}) that Lemma \ref{lem:Da} is equivalent to Proposition \ref{prop:logconcave}.

\begin{Lemma}
\label{lem:Ea}
Let $a(r)$ be a lower semicontinuous function on $(0,1]$.
Then $a(r)$ is concave if and only if
$$\frac{a(r_{1})-a(r_{2})}{r_{1}-r_{2}}\leq \liminf_{r_{3}\to r_{2}-0}\frac{a(r_{3})-a(r_{2})}{r_{3}-r_{2}},$$
holds for any $0< r_{2}<r_{1}\leq 1$.
\end{Lemma}

\subsection{A unified approach to Lemma \ref{lem:GZ_sharp} and Lemma \ref{lem:GZ_sharpA}}\label{sec:unify}

We prove the following Lemma,
which is a unified approach to Lemma \ref{lem:GZ_sharp} and Lemma \ref{lem:GZ_sharpA}
whose various forms already appear in \cite{guan-zhou13p,guan-zhou13ap} etc.:

\begin{Lemma} \label{lem:unify}
Let $B\in(0,+\infty)$ and $t_{0}\geq0$ be arbitrarily given.
Let $D$ be a pseudoconvex domain in $\mathbb{C}^{n}$.
Let $\psi$ be a negative plurisubharmonic function
on $D$.
Let $\varphi$ be a plurisubharmonic function on $D$.
Let $F$ be a holomorphic function on $\{\psi<-t_{0}\}$,
such that
\begin{equation}
\label{equ:20171124a}
\int_{K\cap\{\psi<-t_{0}\}}|F|^{2}<+\infty
\end{equation}
for any compact subset $K$ of $D$,
and
\begin{equation}
\label{equ:20171122a}
\int_{D}\frac{1}{B}\mathbb{I}_{\{-t_{0}-B<\psi<-t_{0}\}}|F|^{2}e^{-\varphi}d\lambda_{n}\leq C<+\infty.
\end{equation}
Then there exists a
holomorphic function $\tilde{F}$ on $D$, such that,
\begin{equation}
\label{equ:3.4}
\begin{split}
\int_{D}&|\tilde{F}-(1-b(\psi))F|^{2}e^{-\varphi+v(\psi)}d\lambda_{n}\leq(1-e^{-(t_{0}+B)})C
\end{split}
\end{equation}
where
$b(t)=\int_{-\infty}^{t}\frac{1}{B}\mathbb{I}_{\{-t_{0}-B< s<-t_{0}\}}ds$,
and $v(t)=\int_{0}^{t}b(s)ds$.
\end{Lemma}

It is clear that $\mathbb{I}_{(-t_{0},+\infty)}\leq b(t)\leq\mathbb{I}_{(-t_{0}-B,+\infty)}$ and $\max\{t,-t_{0}-B\}\leq v(t) \leq\max\{t,-t_{0}\}$.

It suffices to consider the case of Lemma \ref{lem:unify} that
$D$ is a strongly pseudoconvex domain and $\varphi$ and $\psi$ are plurisubharmonic functions on an open set $U$ containing $\bar{D}$,
and $F$ is a holomorphic function on $U\cap\{\psi<-t_{0}\}$ such that $\int_{D\cap\{\psi<-t_{0}\}}|F|^{2}<+\infty$.
In the following remark, we recall some standard steps (see e.g. \cite{siu96,guan-zhou13p,guan-zhou13ap}) to illustrate it.

\begin{Remark}
\label{rem:unify}
It is well-known that there exist strongly pseudoconvex domains $D_{1}\subset\subset\cdots\subset\subset D_{j}\subset\subset D_{j+1}\subset\subset\cdots$
such that $\cup_{j=1}^{+\infty}D_{j}=D$.

If inequality \ref{equ:3.4} holds on any $D_{j}$ and inequality \ref{equ:20171122a} holds on $D$, then we obtain a sequence of holomorphic functions $\tilde{F}_{j}$ on $D_{j}$ such that
\begin{equation}
\begin{split}
&\int_{D_{j}}|\tilde{F}_{j}-(1-b(\psi))F|^{2}e^{-\varphi+v(\psi)}d\lambda_{n}
\\\leq&(1-e^{-(t_{0}+B)})\int_{D_{j}}\frac{1}{B}\mathbb{I}_{\{-t_{0}-B<\psi<-t_{0}\}}|F|^{2}e^{-\varphi}d\lambda_{n}\leq(1-e^{-(t_{0}+B)})C
\end{split}
\end{equation}
is bounded with respect to $j$.
Note that for any given $j$, $e^{-\varphi+v(\psi)}$ has a positive lower bound,
then it follows that for any any given $j$, $\int_{ D_{j}}|\tilde{F}_{j'}-(1-b(\psi))F|^{2}$ is bounded with respect to $j'\geq j$.
Combining with
\begin{equation}
\label{equ:20171123a}
\int_{ D_{j}}|(1-b(\psi))F|^{2}\leq
\int_{D_{j}\cap\{\psi<-t_{0}\}}|F|^{2}<+\infty
\end{equation}
and inequality \ref{equ:3.4},
one can obtain that $\int_{ D_{j}}|\tilde{F}_{j'}|^{2}$ is bounded with respect to $j'\geq j$.

By diagonal method, there exists a subsequence $F_{j''}$ uniformly convergent on any $\bar{D}_{j}$ to a holomorphic function on $D$ denoted by $\tilde{F}$.
Then it follows from inequality \ref{equ:20171123a} and the dominated convergence theorem that
$$\int_{ D_{j}}|\tilde{F}-(1-b(\psi))F|^{2}e^{-\max\{\varphi-v(\psi),-M\}}d\lambda_{n}\leq(1-e^{-(t_{0}+B)})C$$
for any $M>0$,
which implies
$$\int_{ D_{j}}|\tilde{F}-(1-b(\psi))F|^{2}e^{-(\varphi-v(\psi))}d\lambda_{n}\leq(1-e^{-(t_{0}+B)})C,$$
then one can obtain Lemma \ref{lem:unify} when $j$ goes to $+\infty$.
\end{Remark}

For the sake of completeness, we recall some lemmas on $L^{2}$ estimates for some $\bar\partial$ equations, and $\bar\partial^*$ means the Hilbert adjoint operator of
$\bar\partial$.

\begin{Lemma}\label{l:lem3}(see \cite{siu96}, see also \cite{berndtsson})
Let $\Omega\subset\subset \mathbb{C}^{n}$ be a domain with $C^\infty$ boundary $b\Omega$,
$\Phi\in C^{\infty}(\overline \Omega)$,
 Let $\rho$ be a $C^{\infty}$ defining function for $\Omega$
such that $|d\rho|=1$ on $b\Omega$.
Let $\eta$ be a smooth function on $\overline{\Omega}$. For any $(0,1)$-form
$\alpha=\sum_{j=1}^{n}\alpha_{\bar{j}}d\bar z^{j}\in Dom_\Omega(\bar{\partial}^*)\cap C^\infty_{(0,1)}(\overline\Omega)$,
\begin{equation}
\label{guan1}
\begin{split}
&\int_{\Omega}\eta|\bar{\partial}^{*}_{\Phi}\alpha|^{2}e^{-\Phi}d\lambda_{n}
+\int_{\Omega}\eta|\bar{\partial}\alpha|^{2}e^{-\Phi}d\lambda_{n}
=\sum_{i,j=1}^{n}\int_{\Omega}
\eta |\overline\partial_{j}\alpha_{\bar{j}}|^{2}d\lambda_{n}
\\&+\sum_{i,j=1}^{n}
\int_{b\Omega}\eta(\partial_i\bar\partial_{j}\rho)\alpha_{\bar{i}}\overline{{\alpha}_{\bar{j}}}e^{-\Phi}dS
+\sum_{i,j=1}^{n}\int_{\Omega}
\eta(\partial_i\bar \partial_{j}\Phi)\alpha_{\bar{i}}\overline{{\alpha}_{\bar{j}}}e^{-\Phi}d\lambda_{n}
\\&+\sum_{i,j=1}^{n}\int_{\Omega}
-(\partial_i\bar \partial_{j}\eta)\alpha_{\bar{i}}\overline{{\alpha}_{\bar j}}e^{-\Phi}d\lambda_{n}
+2\mathrm{Re}(\bar\partial^*_\Phi\alpha,\alpha\llcorner(\bar\partial\eta)^\sharp )_{\Omega,\Phi},
\end{split}
\end{equation}
where $d\lambda_{n}$ is the Lebesgue measure on $\mathbb{C}^{n}$, and
$\alpha\llcorner(\bar\partial\eta)^\sharp=\sum_{j}\alpha_{\bar{j}}\partial_{j}\eta$.
\end{Lemma}

The symbols and notations can be referred to \cite{guan-zhou13ap}. See also \cite{siu96}, \cite{siu00}, or \cite{Straube}.

\begin{Lemma}\label{l:lem7}(see \cite{berndtsson}, see also \cite{guan-zhou13ap})
Let $\Omega\subset\subset \mathbb{C}^{n}$ be a strictly pseudoconvex domain with $C^\infty$ boundary $b\Omega$ and $\Phi\in C^\infty(\overline{\Omega})$. Let $\lambda$ be a $\bar\partial$ closed smooth form of bidgree $(n,1)$ on $\overline{\Omega}$. Assume the inequality
$$|(\lambda,\alpha)_{\Omega,\Phi}|^{2}\leq C\int_{ \Omega}|\bar{\partial}^{*}_{\Phi}\alpha|^{2}\frac{e^{-\Phi}}{\mu}d\lambda_{n}<\infty,$$
where $\frac{1}{\mu}$ is an integrable positive function on $\Omega$ and
$C$ is a constant, holds for all $(n,1)$-form $\alpha\in Dom_{\Omega}(\bar\partial^*)\cap Ker(\bar\partial)\cap C^\infty_{(n,1)}(\overline \Omega)$. Then there is a solution $u$ to the
equation $\bar{\partial}u=\lambda$ such that
$$\int_{\Omega}|u|^{2}\mu e^{-\Phi}d\lambda_{n}\leq C.$$
\end{Lemma}

\begin{proof}(Proof of Lemma \ref{lem:unify})

For the sake of completeness, let's recall some steps in our proof in
\cite{guan-zhou13p} (see also \cite{guan-zhou13ap,GZopen-effect}) with some slight
modifications in order to prove Lemma \ref{lem:unify}.

By Remark \ref{rem:unify},
one can assume that $D$ is strongly pseudoconvex (with smooth boundary),
and $\psi$ and $\varphi$ are plurisubharmonic on an open set $U$ containing $\bar{D}$,
and $F$ is holomorphic on $U\cap \{\psi<-t_{0}\}$ and
\begin{equation}
\label{equ:20171122e}
\int_{\{\psi<-t_{0}\}\cap D}|F|^{2}<+\infty.
\end{equation}
Then it follows from method of convolution (see e.g. \cite{demailly-book})
that there exist smooth plurisubharmonic functions $\psi_{m}$ and $\varphi_{m}$ on an open set $U\subset\bar{D}$ decreasing convergent to $\psi$ and $\varphi$ respectively,
such that $\sup_{m}\sup_{D}\psi_{m}<0$ and $\sup_{m}\sup_{D}\varphi_{m}<+\infty$.

\

\emph{Step 1: recall some Notations}

\

Let $\varepsilon\in(0,\frac{1}{8}B)$.
Let $\{v_{\varepsilon}\}_{\varepsilon\in(0,\frac{1}{8}B)}$ be a family of smooth increasing convex functions on $\mathbb{R}$,
which are continuous functions on $\mathbb{R}\cup\{-\infty\}$, such that:

 $1).$ $v_{\varepsilon}(t)=t$ for $t\geq-t_{0}-\varepsilon$, $v_{\varepsilon}(t)=constant$ for $t<-t_{0}-B+\varepsilon$;

 $2).$ $v''_{\varepsilon}(t)$ are pointwise convergent to $\frac{1}{B}\mathbb{I}_{(-t_{0}-B,-t_{0})}$, when $\varepsilon\to 0$,
 and $0\leq v''_{\varepsilon}(t)\leq \frac{2}{B}\mathbb{I}_{(-t_{0}-B+\varepsilon,-t_{0}-\varepsilon)}$ for any $t\in \mathbb{R}$;

 $3).$ $v'_{\varepsilon}(t)$ are pointwise convergent to $b(t)$ which is a continuous function on $\mathbb{R}\cup\{-\infty\}$), when $\varepsilon\to 0$, and $0\leq v'_{\varepsilon}(t)\leq1$ for any $t\in \mathbb{R}$.

One can construct the family $\{v_{\varepsilon}\}_{\varepsilon\in(0,\frac{1}{8}B)}$ by the setting
\begin{equation}
\label{equ:20140101}
\begin{split}
v_{\varepsilon}(t):=&\int_{-\infty}^{t}(\int_{-\infty}^{t_{1}}(\frac{1}{B-4\varepsilon}
\mathbb{I}_{(-t_{0}-B+2\varepsilon,-t_{0}-2\varepsilon)}*\rho_{\frac{1}{4}\varepsilon})(s)ds)dt_{1}
\\&-\int_{-\infty}^{0}(\int_{-\infty}^{t_{1}}(\frac{1}{B-4\varepsilon}\mathbb{I}_{(-t_{0}-B+2\varepsilon,
-t_{0}-2\varepsilon)}*\rho_{\frac{1}{4}\varepsilon})(s)ds)dt_{1},
\end{split}
\end{equation}
where $\rho_{\frac{1}{4}\varepsilon}$ is the kernel of convolution satisfying $supp(\rho_{\frac{1}{4}\varepsilon})\subset (-\frac{1}{4}\varepsilon,\frac{1}{4}\varepsilon)$.
Then it follows that
$$v''_{\varepsilon}(t)=\frac{1}{B-4\varepsilon}\mathbb{I}_{(-t_{0}-B+2\varepsilon,-t_{0}-2\varepsilon)}*\rho_{\frac{1}{4}\varepsilon}(t),$$
and
$$v'_{\varepsilon}(t)=\int_{-\infty}^{t}(\frac{1}{B-4\varepsilon}\mathbb{I}_{(-t_{0}-B+2\varepsilon,-t_{0}-2\varepsilon)}
*\rho_{\frac{1}{4}\varepsilon})(s)ds.$$

It suffices to consider the case that
\begin{equation}
\label{equ:20171121a}
\int_{D}\frac{1}{B}\mathbb{I}_{\{-t_{0}-B<\psi<-t_{0}\}}|F|^{2}e^{-\psi-\varphi}d\lambda_{n}<+\infty.
\end{equation}

Let $\eta=s(-v_{\varepsilon}(\psi_{m}))$ and $\phi=u(-v_{\varepsilon}(\psi_{m}))$,
where $s\in C^{\infty}((0,+\infty))$ satisfies $s\geq0$, and
$u\in C^{\infty}((0,+\infty))$, satisfies $\lim_{t\to+\infty}u(t)=0$, such that $u''s-s''>0$, and $s'-u's=1$.
It follows from $\sup_{m}\sum_{D}\psi_{m}<0$ that $\phi=u(-v_{\varepsilon}(\psi_{m}))$ are uniformly bounded
on $D$ with respect to $m$ and $\varepsilon$,
and $u(-v_{\varepsilon}(\psi))$ are uniformly bounded
on $D$ with respect to $\varepsilon$.
Let $\Phi=\phi+\varphi_{m'}$.

\

\emph{Step 2: Solving $\bar\partial-$equation with smooth polar function and smooth weight}

\

Now let $\alpha=\sum^{n}_{j=1}\alpha_{j}d\bar z^{j}\in Dom_{D}
(\bar\partial^*)\cap Ker(\bar\partial)\cap C^\infty_{(0,1)}(\overline {D})$.
By Cauchy-Schwarz inequality, it follows that
\begin{equation}
\label{equ:20131130a}
\begin{split}
2\mathrm{Re}(\bar\partial^*_\Phi\alpha,\alpha\llcorner(\bar\partial\eta)^\sharp )_{D,\Phi}
\geq
&-\int_{D}g^{-1}|\bar{\partial}^{*}_{\Phi}\alpha|^{2}e^{-\Phi}d\lambda_{n}
\\&+
\sum_{j,k=1}^{n}\int_{D}
(-g(\partial_{j} \eta)\bar\partial_{k} \eta )\alpha_{\bar{j} }\overline{{\alpha}_{\bar {k}}}e^{-\Phi}d\lambda_{n}.
\end{split}
\end{equation}

Using Lemma \ref{l:lem3} and inequality \ref{equ:20131130a},
since $s\geq0$ and $\psi_{m}$ is a plurisubharmonic function on $\overline{D}_{v}$,
we get

\begin{equation}
\label{equ:4.1}
\begin{split}
\int_{D}(\eta+g^{-1})|\bar{\partial}^{*}_{\Phi}\alpha|^{2}e^{-\Phi}d\lambda_{n}
\geq&\sum_{j,k=1}^{n}\int_{D}
(-\partial_{j}\bar{\partial}_{k}\eta+\eta\partial_{j}\bar{\partial}_{k}\Phi-g(\partial_{j} \eta)\bar\partial_{k} \eta )\alpha_{\bar{j} }\overline{{\alpha}_{\bar{k}}}e^{-\Phi}d\lambda_{n}
\\\geq&\sum_{j,k=1}^{n}\int_{D}
(-\partial_{j}\bar{\partial}_{k}\eta+\eta\partial_{j}\bar{\partial}_{k}\phi-g(\partial_{j} \eta)\bar\partial_{k} \eta )\alpha_{\bar{j} }\overline{{\alpha}_{\bar {k}}}e^{-\Phi}d\lambda_{n},
\end{split}
\end{equation}
where $g$ is a positive continuous function on $D$.
We need some calculations to determine $g$.

We have

\begin{equation}
\label{}
\begin{split}
&\partial_{j}\bar{\partial}_{k}\eta=-s'(-v_{\varepsilon}(\psi_{m}))\partial_{j}\bar{\partial}_{k}(v_{\varepsilon}(\psi_{m}))
+s''(-v_{\varepsilon}(\psi_{m}))\partial_{j}v_{\varepsilon}(\psi_{m})
\bar{\partial}_{k}v_{\varepsilon}(\psi_{m}),
\end{split}
\end{equation}

and
\begin{equation}
\label{}
\begin{split}
&\partial_{j}\bar{\partial}_{k}\phi=-u'(-v_{\varepsilon}(\psi_{m}))\partial_{j}\bar{\partial}_{k}v_{\varepsilon}(\psi_{m})
+
u''(-v_{\varepsilon}(\psi_{m}))\partial_{j}v_{\varepsilon}(\psi_{m})\bar{\partial}_{k}v_{\varepsilon}(\psi_{m})
\end{split}
\end{equation}

for any $j,k$ ($1\leq j,k\leq n$).

We have

\begin{equation}
\label{}
\begin{split}
&\sum_{1\leq j,k\leq n}(-\partial_{j}\bar{\partial}_{k}\eta+\eta\partial_{j}\bar{\partial}_{k}\phi-g(\partial_{j} \eta)
\bar\partial_{k} \eta )\alpha_{\bar{j} }\overline{{\alpha}_{\bar{k}}}
\\=&(s'-su')\sum_{1\leq j,k\leq n}\partial_{j}\bar{\partial}_{k}v_{\varepsilon}(\psi_{m})\alpha_{\bar{j} }\overline{{\alpha}_{\bar{k}}}
\\+&((u''s-s'')-gs'^{2})\sum_{1\leq j,k\leq n}\partial_{j}
(-v_{\varepsilon}(\psi_{m}))\bar{\partial}_{k}(-v_{\varepsilon}(\psi_{m}))\alpha_{\bar{j} }\overline{{\alpha}_{\bar{k}}}
\\=&(s'-su')\sum_{1\leq j,k\leq n}(v'_{\varepsilon}(\psi_{m})\partial_{j}\bar{\partial}_{k}\psi_{m}+v''_{\varepsilon}(\psi_{m})
\partial_{j}(\psi_{m})\bar{\partial}_{k}(\psi_{m}))\alpha_{\bar{j} }\overline{{\alpha}_{\bar{k}}}
\\+&((u''s-s'')-gs'^{2})\sum_{1\leq j,k\leq n}\partial_{j}
(-v_{\varepsilon}(\psi_{m}))\bar{\partial}_{k}(-v_{\varepsilon}(\psi_{m}))\alpha_{\bar{j} }\overline{{\alpha}_{\bar{k}}}.
\end{split}
\end{equation}
We omit composite item $-v_{\varepsilon}(\psi_{m})$ after $s'-su'$ and $(u''s-s'')-gs'^{2}$ in the above equalities.

Let $g=\frac{u''s-s''}{s'^{2}}(-v_{\varepsilon}(\psi_{m}))$.
It follows that $\eta+g^{-1}=(s+\frac{s'^{2}}{u''s-s''})(-v_{\varepsilon}(\psi_{m}))$.

Because of $v'_{\varepsilon}\geq 0$  and $s'-su'=1$, using inequalities \ref{equ:4.1}, we have
\begin{equation}
\label{equ:3.1}
\begin{split}
\int_{D}(\eta+g^{-1})|\bar{\partial}^{*}_{\Phi}\alpha|^{2}e^{-\Phi}d\lambda_{n}
\geq\int_{D}(v''_{\varepsilon}\circ{\psi_{m}})
\big|\alpha\llcorner(\bar \partial \psi_{m})^\sharp\big|^2e^{-\Phi}d\lambda_{n}.
\end{split}
\end{equation}

As $F$ is holomorphic on $\{\psi<-t_{0}\}\supset\supset Supp(v'_{\varepsilon}(\psi_{m}))$,
then $\lambda:=\bar{\partial}[(1-v'_{\varepsilon}(\psi_{m})){F}]$
is well-defined and smooth on $D$.
By the definition of contraction, Cauchy-Schwarz inequality and inequality \ref{equ:3.1},
it follows that
\begin{equation}
\begin{split}
|(\lambda,\alpha)_{D,\Phi}|^{2}=
&|(v''_{\varepsilon}(\psi_{m})\bar\partial\psi_{m} F,\alpha)_{D,\Phi}|^{2}
\\=&|(v''_{\varepsilon}(\psi_{m})F,\alpha\llcorner(\bar\partial\psi_{m})^\sharp\big)_{D,\Phi}|^{2}
\\\leq&\int_{D}
v''_{\varepsilon}(\psi_{m})| F|^2e^{-\Phi}d\lambda_{n}\int_{D}v''_{\varepsilon}(\psi_{m})
\big|\alpha\llcorner(\bar\partial\psi_{m})^\sharp\big|^2e^{-\Phi}d\lambda_{n}
\\\leq&
(\int_{D}
v''_{\varepsilon}(\psi_{m})| F|^2e^{-\Phi}d\lambda_{n})
(\int_{ D}(\eta+g^{-1})|\bar{\partial}^{*}_{\Phi}\alpha|^{2}e^{-\Phi}d\lambda_{n}).
\end{split}
\end{equation}

Let $\mu:=(\eta+g^{-1})^{-1}$. Using Lemma \ref{l:lem7},
we have locally $L^{1}$ function $u_{m,m',\varepsilon}$ on $D$ such that $\bar{\partial}u_{m,m',\varepsilon}=\lambda$,
and
\begin{equation}
 \label{equ:3.2}
 \begin{split}
 &\int_{ D}|u_{m,m',\varepsilon}|^{2}(\eta+g^{-1})^{-1} e^{-\Phi}d\lambda_{n}
  \leq\int_{D}(v''_{\varepsilon}(\psi_{m}))| F|^2e^{-\Phi}d\lambda_{n}.
  \end{split}
\end{equation}

Assume that we can choose $\eta$ and $\phi$ such that $e^{v_{\varepsilon}\circ\psi_{m}}e^{\phi}=(\eta+g^{-1})^{-1}$.
Then inequality \ref{equ:3.2} becomes
\begin{equation}
 \label{equ:20171122b}
 \begin{split}
 &\int_{ D}|u_{m,m',\varepsilon}|^{2}e^{v_{\varepsilon}(\psi_{m})-\varphi_{m'}}d\lambda_{n}
  \leq\int_{D}v''_{\varepsilon}(\psi_{m})| F|^2e^{-\phi-\varphi_{m'}}d\lambda_{n}.
  \end{split}
\end{equation}

Let $F_{m,m',\varepsilon}:=-u_{m,m',\varepsilon}+(1-v'_{\varepsilon}(\psi_{m})){F}$.
Then inequality \ref{equ:20171122b} becomes
\begin{equation}
 \label{equ:20171122c}
 \begin{split}
 &\int_{ D}|F_{m,m',\varepsilon}-(1-v'_{\varepsilon}(\psi_{m})){F}|^{2}e^{v_{\varepsilon}(\psi_{m})-\varphi_{m'}}d\lambda_{n}
  \leq\int_{D}(v''_{\varepsilon}(\psi_{m}))| F|^2e^{-\phi-\varphi_{m'}}d\lambda_{n}.
  \end{split}
\end{equation}

\

\emph{Step 3: Singular polar function and smooth weight}

\

As $\sup_{m,\varepsilon}|\phi|=\sup_{m,\varepsilon}|u(-v_{\varepsilon}(\psi_{m}))|<+\infty$ and $\varphi_{m'}$ is continuous on $\bar{D}$,
then $\sup_{m,\varepsilon}e^{-\phi-\varphi_{m'}}<+\infty$.
Note that
$$v''_{\varepsilon}(\psi_{m})| F|^2e^{-\phi-\varphi_{m'}}\leq\frac{2}{B}\mathbb{I}_{\{\psi<-t_{0}\}}| F|^{2}\sup_{m,\varepsilon}e^{-\phi-\varphi_{m'}}$$
on $D$,
then it follows from inequality \ref{equ:20171122e} and the dominated convergence theorem that
\begin{equation}
\label{equ:20171122f}
 \lim_{m\to+\infty}\int_{D}v''_{\varepsilon}(\psi_{m})| F|^2e^{-\phi-\varphi_{m'}}d\lambda_{n}=
 \int_{D}v''_{\varepsilon}(\psi)| F|^2e^{-u(-v_{\varepsilon}(\psi))-\varphi_{m'}}d\lambda_{n}
\end{equation}

Note that $\inf_{m}\inf_{D}e^{v_{\varepsilon}(\psi_{m})-\varphi_{m'}}>0$,
then it follows from inequality \ref{equ:20171122c} and \ref{equ:20171122f}
that $\sup_{m}\int_{ D}|F_{m,m',\varepsilon}-(1-v'_{\varepsilon}(\psi_{m})){F}|^{2}<+\infty$.
Note that
\begin{equation}
\label{equ:20171122g}
|(1-v'_{\varepsilon}(\psi_{m}))F|\leq |\mathbb{I}_{\{\psi<-t_{0}\}}F|,
\end{equation}
then it follows from inequality \ref{equ:20171122e}
that $\sup_{m}\int_{ D}|F_{m,m',\varepsilon}|^{2}<+\infty$,
which implies that there exists a subsequence of $\{F_{m,m',\varepsilon}\}_{m}$
(also denoted by $F_{m,m',\varepsilon}$) compactly convergent to a holomorphic $F_{m',\varepsilon}$ on $D$.

Note that $v_{\varepsilon}(\psi_{m})-\varphi_{m'}$ are uniformly bounded on $D$ with respect to $m$,
then it follows from
$|F_{m,m',\varepsilon}-(1-v'_{\varepsilon}(\psi_{m})){F}|^{2}\leq
 2(|F_{m,m',\varepsilon}|^{2}+ |(1-v'_{\varepsilon}(\psi_{m})){F}|^{2}
\leq  2(|F_{m,m',\varepsilon}|^{2}+  |\mathbb{I}_{\{\psi<-t_{0}\}}F^{2}|)$
and the dominated convergence theorem that
\begin{equation}
 \label{equ:20171122d}
 \begin{split}
 \lim_{m\to+\infty}&\int_{K}|F_{m,m',\varepsilon}-(1-v'_{\varepsilon}(\psi_{m})){F}|^{2}e^{v_{\varepsilon}(\psi_{m})-\varphi_{m'}}d\lambda_{n}
  \\=&\int_{K}|F_{m',\varepsilon}-(1-v'_{\varepsilon}(\psi)){F}|^{2}e^{v_{\varepsilon}(\psi)-\varphi_{m'}}d\lambda_{n}
  \end{split}
\end{equation}
holds for any compact subset $K$ on $D$.
Combining with inequality \ref{equ:20171122c} and \ref{equ:20171122f},
one can obtain that
\begin{equation}
 \label{equ:20171122h}
\int_{K}|F_{m',\varepsilon}-(1-v'_{\varepsilon}(\psi)){F}|^{2}e^{v_{\varepsilon}(\psi)-\varphi_{m'}}d\lambda_{n}
\leq
\int_{D}v''_{\varepsilon}(\psi)| F|^2e^{-u(-v_{\varepsilon}(\psi))-\varphi_{m'}}d\lambda_{n},
\end{equation}
which implies
\begin{equation}
 \label{equ:20171122i}
\int_{D}|F_{m',\varepsilon}-(1-v'_{\varepsilon}(\psi)){F}|^{2}e^{v_{\varepsilon}(\psi)-\varphi_{m'}}d\lambda_{n}
\leq
\int_{D}v''_{\varepsilon}(\psi)| F|^2e^{-u(-v_{\varepsilon}(\psi))-\varphi_{m'}}d\lambda_{n},
\end{equation}

\

\emph{Step 4: Nonsmooth cut-off function}

\

Note that
$\sup_{\varepsilon}\sup_{D}e^{-u(-v_{\varepsilon}(\psi))-\varphi_{m'}}<+\infty,$
and
$$v''_{\varepsilon}(\psi)| F|^2e^{-u(-v_{\varepsilon}(\psi))-\varphi_{m'}}\leq
\frac{2}{B}\mathbb{I}_{\{-t_{0}-B<\psi<-t_{0}\}}| F|^2\sup_{\varepsilon}\sup_{D}e^{-u(-v_{\varepsilon}(\psi))-\varphi_{m'}},$$
then it follows from inequality \ref{equ:20171122e} and the dominated convergence theorem that
\begin{equation}
\label{equ:20171122j}
\begin{split}
&\lim_{\varepsilon\to0}\int_{D}v''_{\varepsilon}(\psi)| F|^2e^{-u(-v_{\varepsilon}(\psi))-\varphi_{m'}}d\lambda_{n}
\\=&\int_{D}\frac{1}{B}\mathbb{I}_{\{-t_{0}-B<\psi<-t_{0}\}}|F|^2e^{-u(-v(\psi))-\varphi_{m'}}d\lambda_{n}
\\\leq&(\sup_{D}e^{-u(-v(\psi))})\int_{D}\frac{1}{B}\mathbb{I}_{\{-t_{0}-B<\psi<-t_{0}\}}|F|^2e^{-\varphi_{m'}}<+\infty.
\end{split}
\end{equation}

Note that $\inf_{\varepsilon}\inf_{D}e^{v_{\varepsilon}(\psi)-\varphi_{m'}}>0$,
then it follows from inequality \ref{equ:20171122i} and \ref{equ:20171122j} that
$\sup_{\varepsilon}\int_{D}|F_{m',\varepsilon}-(1-v'_{\varepsilon}(\psi)){F}|^{2}<+\infty.$
Combining with
\begin{equation}
\label{equ:20171122k}
\sup_{\varepsilon}\int_{D}|(1-v'_{\varepsilon}(\psi)){F}|^{2}\leq\int_{D}\mathbb{I}_{\{\psi<-t_{0}\}}|F^{2}|<+\infty,
\end{equation}
one can obtain that $\sup_{\varepsilon}\int_{D}|F_{m',\varepsilon}|^{2}<+\infty$,
which implies that
there exists a subsequence of $\{F_{m',\varepsilon}\}_{\varepsilon\to0}$ (also denoted by $\{F_{m',\varepsilon}\}_{\varepsilon\to0}$)
compactly convergent to a holomorphic function on $D$ denoted by $F_{m'}$.

Note that $\sup_{\varepsilon}\sup_{D}e^{v_{\varepsilon}(\psi)-\varphi_{m'}}<+\infty$ and
$|F_{m',\varepsilon}-(1-v'_{\varepsilon}(\psi)){F}|^{2}\leq 2(|F_{m',\varepsilon}|^{2}+|\mathbb{I}_{\{\psi<-t_{0}\}}F|^{2})$,
then it follows from inequality \ref{equ:20171122k} and the dominated convergence theorem on any given $K\subset\subset D$
$($with dominant function $2(\sup_{\varepsilon}\sup_{K}(|F_{m',\varepsilon}|^{2})+\mathbb{I}_{\{\psi<-t_{0}\}}|F|^{2})\sup_{\varepsilon}\sup_{D}e^{v_{\varepsilon}(\psi)-\varphi_{m'}})$
that
\begin{equation}
\label{equ:20171122l}
\begin{split}
&\lim_{\varepsilon\to0}\int_{K}|F_{m',\varepsilon}-(1-v'_{\varepsilon}(\psi)){F}|^{2}e^{v_{\varepsilon}(\psi)-\varphi_{m'}}d\lambda_{n}
\\=&\int_{K}|F_{m'}-(1-b(\psi)){F}|^{2}e^{v(\psi)-\varphi_{m'}}d\lambda_{n}.
\end{split}
\end{equation}
Combining with inequality \ref{equ:20171122j} and \ref{equ:20171122i}, one can obtain that
\begin{equation}
\label{equ:20171122m}
\begin{split}
&\int_{K}|F_{m'}-(1-b(\psi)){F}|^{2}e^{v(\psi)-\varphi_{m'}}d\lambda_{n}
\\\leq&(\sup_{D}e^{-u(-v(\psi))})\int_{D}\frac{1}{B}\mathbb{I}_{\{-t_{0}-B<\psi<-t_{0}\}}|F|^2e^{-\varphi_{m'}}
\end{split}
\end{equation}
which implies
\begin{equation}
\label{equ:20171122n}
\begin{split}
&\int_{D}|F_{m'}-(1-b(\psi)){F}|^{2}e^{v(\psi)-\varphi_{m'}}d\lambda_{n}
\\\leq&(\sup_{D}e^{-u(-v(\psi))})\int_{D}\frac{1}{B}\mathbb{I}_{\{-t_{0}-B<\psi<-t_{0}\}}|F|^2e^{-\varphi_{m'}}.
\end{split}
\end{equation}

\

\emph{Step 5: Singular weight}

\

Note that
\begin{equation}
\label{equ:20171122o}
\int_{D}\frac{1}{B}\mathbb{I}_{\{-t_{0}-B<\psi<-t_{0}\}}|F|^2e^{-\varphi_{m'}}\leq\int_{D}\frac{1}{B}\mathbb{I}_{\{-t_{0}-B<\psi<-t_{0}\}}|F|^{2}e^{-\varphi}<+\infty,
\end{equation}
and $\sup_{D}e^{-u(-v(\psi))}<+\infty$,
then it from \ref{equ:20171122n} that
$$\sup_{m'}\int_{D}|F_{m'}-(1-b(\psi)){F}|^{2}e^{v(\psi)-\varphi_{m'}}d\lambda_{n}<+\infty.$$
Combining with $\inf_{m'}\inf_{D}e^{v(\psi)-\varphi_{m'}}>0$,
one can obtain that
$\sup_{m'}\int_{D}|F_{m'}-(1-b(\psi)){F}|^{2}d\lambda_{n}<+\infty$.
Note that
\begin{equation}
\label{equ:20171122p}
\int_{D}|(1-b(\psi)){F}|^{2}d\lambda_{n}\leq\int_{D}|\mathbb{I}_{\{\psi<-t_{0}\}}F|^{2}d\lambda_{n} <+\infty.
\end{equation}
Then $\sup_{m'}\int_{D}|F_{m'}|^{2}d\lambda_{n}<+\infty$,
which implies that there exists a compactly convergent subsequence of $\{F_{m'}\}$ denoted by $\{F_{m''}\}$,
which is convergent a holomorphic function $\tilde{F}$ on $D$.

Note that $\sup_{m'}\sup_{D}e^{v(\psi)-\varphi_{m'}}<+\infty$,
then it follows from inequality \ref{equ:20171122p} and the
dominated convergence theorem on any given compact subset $K$ of $D$ $($with dominant function $2[\sup_{m''}\sup_{K}(|F_{m''}|^{2})+\mathbb{I}_{\{\psi<-t_{0}\}}|F|^{2}]\sup_{D}e^{v(\psi)-\varphi_{m'}}$ $)$ that
\begin{equation}
\label{equ:20171122q}
\begin{split}
&\lim_{m''\to+\infty}\int_{K}|F_{m''}-(1-b(\psi)){F}|^{2}e^{v(\psi)-\varphi_{m'}}d\lambda_{n}
\\=&\int_{K}|\tilde{F}-(1-b(\psi)){F}|^{2}e^{v(\psi)-\varphi_{m'}}d\lambda_{n}.
\end{split}
\end{equation}
Note that for any $m''\geq m'$, $\varphi_{m'}\leq\varphi_{m''}$ holds,
then it follows from inequality \ref{equ:20171122n} and \ref{equ:20171122o}
that
\begin{equation}
\label{equ:20171122r}
\begin{split}
&\lim_{m''\to+\infty}\int_{K}|F_{m''}-(1-b(\psi)){F}|^{2}e^{v(\psi)-\varphi_{m'}}d\lambda_{n}
\\\leq&
\limsup_{m''\to+\infty}\int_{K}|F_{m''}-(1-b(\psi)){F}|^{2}e^{v(\psi)-\varphi_{m''}}d\lambda_{n}
\\\leq&
\limsup_{m''\to+\infty}(\sup_{D}e^{-u(-v(\psi))})\int_{D}\frac{1}{B}\mathbb{I}_{\{-t_{0}-B<\psi<-t_{0}\}}|F|^2e^{-\varphi_{m''}}
\\\leq&
(\sup_{D}e^{-u(-v(\psi))})C<+\infty.
\end{split}
\end{equation}
Combining with equality \ref{equ:20171122q},
one can obtain that
$$\int_{K}|\tilde{F}-(1-b(\psi)){F}|^{2}e^{v(\psi)-\varphi_{m'}}d\lambda_{n}\leq(\sup_{D}e^{-u(-v(\psi))})C,$$
for any compact subset of $D$,
which implies
$$\int_{D}|\tilde{F}-(1-b(\psi)){F}|^{2}e^{v(\psi)-\varphi_{m'}}d\lambda_{n}\leq(\sup_{D}e^{-u(-v(\psi))})C.$$
When $m'\to+\infty$,
it follows from Levi's Theorem that
\begin{equation}
\label{equ:20171122s}
\begin{split}
\int_{D}|\tilde{F}-(1-b(\psi)){F}|^{2}e^{v(\psi)-\varphi}d\lambda_{n}\leq(\sup_{D}e^{-u(-v(\psi))})C.
\end{split}
\end{equation}

\

\emph{Step 6: ODE system}

\

It suffices to find $\eta$ and $\phi$ such that
$(\eta+g^{-1})=e^{-\psi_{m}}e^{-\phi}$ on $D$.
As $\eta=s(-v_{\varepsilon}(\psi_{m}))$ and $\phi=u(-v_{\varepsilon}(\psi_{m}))$,
we have $(\eta+g^{-1}) e^{v_{\varepsilon}(\psi_{m})}e^{\phi}=(s+\frac{s'^{2}}{u''s-s''})e^{-t}e^{u}\circ(-v_{\varepsilon}(\psi_{m}))$.

Summarizing the above discussion about $s$ and $u$, we are naturally led to a
system of ODEs (see \cite{guan-zhou12,guan-zhou13p,guan-zhou13ap,GZopen-effect}):
\begin{equation}
\label{GZ}
\begin{split}
&1).\,\,(s+\frac{s'^{2}}{u''s-s''})e^{u-t}=1, \\
&2).\,\,s'-su'=1,
\end{split}
\end{equation}
where $t\in[0,+\infty)$.

It is not hard to solve the ODE system \ref{GZ} and get $u=-\log(1-e^{-t})$ and
$s=\frac{t}{1-e^{-t}}-1$.
It follows that $s\in C^{\infty}((0,+\infty))$ satisfies $s\geq0$, $\lim_{t\to+\infty}u(t)=0$ and
$u\in C^{\infty}((0,+\infty))$ satisfies $u''s-s''>0$.

As $u=-\log(1-e^{-t})$ is decreasing with respect to $t$,
then it follows from $0\geq v(t)\geq\max\{t,-t_{0}-B_{0}\}\geq -t_{0}-B_{0}$ for any $t\leq0$
that
\begin{equation}
\begin{split}
\sup_{D}e^{-u(-v(\psi))}
\leq\sup_{t\in(0,t_{0}+B]}e^{-u(t)}
=\sup_{t\in(0,t_{0}+B]}(1-e^{-t})=1-e^{-(t_{0}+B)},
\end{split}
\end{equation}
therefore we are done.
Thus we prove Lemma \ref{lem:unify}.

\end{proof}

\vspace{.1in} {\em Acknowledgements}. The author would like to
thank Professor Xiangyu Zhou for helpful discussions and encouragements.
The author would also like to
thank the hospitality of Beijing International Center for Mathematical Research.

\bibliographystyle{references}
\bibliography{xbib}

\begin{thebibliography}{100}
\bibitem{berndtsson}B. Berndtsson, The extension theorem of Ohsawa-Takegoshi and the theorem of
Donnelly-Fefferman, Ann. L'Inst. Fourier (Grenoble) 46 (1996), no.
4, 1083--1094.
\bibitem{berndtsson13}B. Berndtsson, The openness conjecture for plurisubharmonic functions, arXiv:1305.5781.
\bibitem{demailly-note2000}J-P. Demailly, Multiplier ideal sheaves and analytic methods in algebraic geometry. School on Vanishing Theorems and Effective Results in Algebraic Geometry (Trieste, 2000), 1--148, ICTP Lect. Notes, 6, Abdus Salam Int. Cent. Theoret. Phys., Trieste, 2001.
\bibitem{demailly2010}J.-P. Demailly, Analytic Methods in Algebraic Geometry, Higher Education Press, Beijing, 2010.
\bibitem{demailly-book}J.-P. Demailly, Complex analytic and differential geometry, electronically accessible
at http://www-fourier.ujf-grenoble.fr/~demailly/books.html.
\bibitem{DEL00}J-P. Demailly, L. Ein, and R. Lazarsfeld, A subadditivity property of multiplier
ideals, Michigan Math. J. 48 (2000), 137--156.
\bibitem{D-K01}J-P. Demailly, J. Koll\'{a}r,
Semi-continuity of complex singularity exponents and K\"{a}hler-Einstein metrics on Fano orbifolds.
Ann. Sci. \'{E}cole Norm. Sup. (4) 34 (2001), no. 4, 525--556.
\bibitem{D-P03}J-P. Demailly, T. Peternell, A Kawamata-Viehweg vanishing theorem on compact K\"{a}hler manifolds. J. Differential Geom.  63  (2003),  no. 2, 231--277.
\bibitem{FM05v}C. Favre and M. Jonsson, Valuative analysis of planar plurisubharmonic functions, Invent. Math. 162 (2005), no. 2, 271--311.
\bibitem{FM05j}C. Favre and M. Jonsson, Valuations and multiplier ideals, J. Amer. Math. Soc. 18 (2005), no. 3, 655--684.
\bibitem{G-R}H. Grauert and R. Remmert, \emph{Coherent Analytic Sheaves}, Grundlehren der mathematischen Wissenchaften, 265, Springer-Verlag, Berlin, 1984.
\bibitem{guan-zhou12}Q.A. Guan and X.Y. Zhou, Optimal constant problem in the $L^{2}$ extension theorem, C. R. Acad. Sci. Paris. Ser. I. 350 (2012), no. 15--16, 753--756.
\bibitem{guan-zhou13p}Q.A. Guan and X.Y. Zhou, Optimal constant in an $L^2$ extension problem and a proof of a conjecture of Ohsawa,  Sci. China Math., 2015, 58(1): 35--59.
\bibitem{guan-zhou13ap}Q.A. Guan and X.Y. Zhou, A solution of an $L^{2}$ extension problem with an optimal estimate and applications,
Ann. of Math. (2) 181 (2015), no. 3, 1139--1208.
\bibitem{GZopen-c}Q.A. Guan and X.Y. Zhou, A proof of Demailly's strong openness conjecture,
Ann. of Math. (2) 182 (2015), no. 2, 605--616. See also arXiv:1311.3781.
\bibitem{GZopen-effect}Q.A. Guan and X.Y. Zhou, Effectiveness of Demailly's strong openness conjecture and related problems, Invent. Math. 202 (2015), no. 2, 635--676.
\bibitem{Gue12}H. Guenancia, Toric plurisubharmonic functions and analytic adjoint ideal sheaves,
Math. Z. 271 (2012), no. 3--4, 1011--1035.
\bibitem{Hiep14}P. H. Hiep, The weighted log canonical threshold. C. R. Math. Acad. Sci. Paris 352 (2014), no. 4, 283--288.
\bibitem{JM12}M. Jonsson and M. Musta\c{t}\u{a}, Valuations and asymptotic invariants for sequences of ideals, Annales de l'Institut Fourier A. 2012, vol. 62, no.6, pp. 2145--2209.
\bibitem{JM13}M. Jonsson and M. Musta\c{t}\u{a}, An algebraic approach to the openness conjecture of Demailly and Koll\'{a}r,
J. Inst. Math. Jussieu (2013), 1--26.
\bibitem{Ko92}J. Koll\'{a}r (with 14 coauthors): Flips and Abundance for Algebraic Threefolds; Ast\'{e}risque Vol. 211 (1992).
\bibitem{Laz04}Lazarsfeld R. Positivity in algebraic geometry. I. Classical setting: line bundles and linear series;
II. Positivity for vector bundles, and multiplier ideals. Ergebnisse der Mathematik und ihrer Grenzgebiete. 3. Folge.
 A Series of Modern Surveys in Mathematics, 48, 49. Springer-Verlag, Berlin, 2004.
\bibitem{Lempert14}L. Lempert, Modules of square integrable holomorphic germs, arXiv:1404.0407v2.
\bibitem{Nadel90}A. Nadel, Multiplier ideal sheaves and K\"{a}hler-Einstein metrics of positive scalar curvature.
Ann. of Math. (2) 132 (1990), no. 3, 549--596.
\bibitem{ohsawa3}T. Ohsawa, On the extension of $L^{2}$ holomorphic functions. III.
Negligible weights. Math. Z. 219 (1995), no. 2, 215--225.
\bibitem{Sho92}V. Shokurov: 3-fold log flips; Izv. Russ. Acad. Nauk Ser. Mat. Vol. 56 (1992) 105--203.
\bibitem{Si}Nessim Sibony: Quelques probl\`emes de prolongement de courants en
analyse complexe. (French) [Some extension problems for currents in
complex analysis] Duke Math. J. 52 (1985), no. 1, 157--197.
\bibitem{siu74}Y.T. Siu, Analyticity of sets associated to Lelong numbers and the extension of closed positive currents. Invent. Math. 27 (1974), 53--156.
\bibitem{siu96}Y.T. Siu, The Fujita conjecture and the extension theorem of Ohsawa-Takegoshi,
Geometric Complex Analysis, Hayama. World Scientific (1996),
577--592.
\bibitem{siu00}Y.T. Siu, Extension of twisted pluricanonical sections with plurisubharmonic weight and invariance of semipositively twisted plurigenera for manifolds not necessarily of general type. Complex geometry (G\"{o}ttingen, 2000), 223--277, Springer, Berlin, (2002).
\bibitem{siu05}Y.T. Siu, Multiplier ideal sheaves in complex and algebraic geometry. Sci. China Ser. A 48
(2005), suppl., 1--31.
\bibitem{siu09} Y.T. Siu, Dynamic multiplier ideal sheaves and the construction of rational curves in Fano
manifolds. Complex analysis and digital geometry, 323--360, Acta Univ. Upsaliensis Skr.
Uppsala Univ. C Organ. Hist., 86, Uppsala Universitet, Uppsala, 2009.
\bibitem{Straube}E. Straube, Lectures on the $L^2$-Sobolev Theory of the $\bar\partial$-Neumann Problem.
ESI Lectures in Mathematics and Physics. Z¨¹rich: European Mathematical
Society (2010).
\bibitem{tian87}G. Tian, On K\"{a}hler-Einstein metrics on certain K\"{a}hler manifolds with $C_{1}(M)>0$, Invent. Math.  89  (1987),  no. 2, 225--246.
\end{thebibliography}

\end{document}